\documentclass[letterpaper]{amsart}
\usepackage{graphicx} 
\usepackage{amsmath,amssymb,amsthm, epsfig}
\usepackage[pdftex,colorlinks]{hyperref}

\hypersetup{linkcolor = black, citecolor=blue}

\usepackage{mathtools}
\usepackage{mathtools}
\usepackage{mathrsfs}
\usepackage{tikz,tkz-base,tkz-fct}
\usepackage{comment}
\usepackage{pstricks}
\usepackage[margin=1in]{geometry}
\usepackage{stackrel}
\usepackage{booktabs}
\usepackage{verbatim}
\usepackage{esint}

\newcommand{\intav}[1]{\mathchoice {\mathop{\vrule width 6pt height 3 pt depth  -2.5pt
\kern -8pt \intop}\nolimits_{\kern -6pt#1}} {\mathop{\vrule width
5pt height 3  pt depth -2.6pt \kern -6pt \intop}\nolimits_{#1}}
{\mathop{\vrule width 5pt height 3 pt depth -2.6pt \kern -6pt
\intop}\nolimits_{#1}} {\mathop{\vrule width 5pt height 3 pt depth
-2.6pt \kern -6pt \intop}\nolimits_{#1}}}
\newcommand{\defeq}{\mathrel{\mathop:}=}

\newtheorem{theorem}{Theorem}[section]

\newtheorem{lemma}[theorem]{Lemma}

\newtheorem{corollary}[theorem]{Corollary}

\theoremstyle{definition}

\newtheorem{definition}[theorem]{Definition}

\theoremstyle{remark}

\newtheorem{remark}[theorem]{Remark}

\numberwithin{equation}{section}

\title[Fully Nonlinear Evolution equations: From Critical Integrability to Schauder Estimates]{Fully Nonlinear Evolution equations: From Critical Integrability to Schauder Estimates}

\author[Bessa]{Junior da Silva Bessa}
\address{Universidade Estadual de Campinas (UNICAMP),
Instituto de Matem\'{a}tica, Estat\'{i}stica e Computa\c{c}\~{a}o Cient\'{i}fica (IMECC),
Departamento de Matem\'{a}tica,
Bar\~{a}o Geraldo, Campinas, SP, Brazil}
\email{jbessa@unicamp.br}
\author[Oliveira]{ José Erivamberto L. Oliveira} 
\address{Universidade Estadual do Ceará - UECE, Faculdade de Filosofia Dom Aureliano Matos - FAFIDAM, Departamento de Matem\'{a}tica, Limoeiro do Norte - CE, Brazil}
\email{erivamberto.oliveira@uece.br}
\author[Regis]{Patrícia Renata Pereira Regis}
\address{Universidade Estadual do Ceará - UECE, Faculdade de Filosofia Dom Aureliano Matos - FAFIDAM, Departamento de Matem\'{a}tica, Limoeiro do Norte - CE, Brazil}
\email{patricia.renata@uece.br}



\thanks{Acknowledgements. Junior da Silva Bessa has been supported by FAPESP-Brazil under Grant No. 2023/18447-3.}
\date{}

\begin{document}

\begin{abstract}
\noindent We investigate the sharp regularity up to the boundary for viscosity solutions of fully nonlinear parabolic equations with oblique derivative boundary conditions given by
\begin{equation*}
\left\{
\begin{array}{rclcl}
F(D^2u,x,t) - u_{t} &=& f(x,t) & \text{in} & Q_{1}^{+}, \\
\beta(x,t) \cdot Du &=& g(x,t) & \text{on} & Q_{1}^{*}.
\end{array}
\right.
\end{equation*}
The regularity theory is developed according to the integrability of the source term, the smoothness of the boundary data, and the oscillation of the coefficients of the operator \(F\), using a compactness method combined with polynomial approximation. In the borderline case \(f\in L^{n+2}\), we obtain Log-Lipschitz continuity of solutions up to the boundary. Under stronger integrability, specifically when \(f\in L^{p}\) for some \(p>n+2\), we establish optimal \(C^{1+\alpha',\,\frac{1+\alpha'}{2}}\) boundary estimates. At a higher regularity level, we prove Schauder-type estimates under appropriate assumptions on the operator \(F\) and the boundary data. As a byproduct, we obtain parabolic \(C^{1,\mathrm{Log\text{-}Lip}}\) regularity in the critical borderline case where the source term belongs to BMO space.
\end{abstract}
\maketitle

\tableofcontents

\section{Introduction}

\noindent
In this paper, we establish a sharp regularity theory for viscosity solutions of fully nonlinear evolution equations with oblique derivative boundary conditions of the form
\begin{equation}\label{1.1}
\left\{
\begin{array}{rclcl}
F(D^2u,x,t) - u_{t} &=& f(x,t) & \text{in} & Q_{1}^{+}, \\
\beta(x,t) \cdot Du &=& g(x,t) & \text{on} & Q_{1}^{*},
\end{array}
\right.
\end{equation}
where $F$ is a uniformly parabolic operator (see {\bf (A1)}), $f$, $g$ are continuous functions and $\beta:Q_{1}^{*}\to \mathbb{R}^{n}$ is a continuous vector field satisfying the oblique and normalization conditions, namely, there exists a positive constant $\mu_{0}$ such that
\begin{equation}\label{1.2}
\beta_{n}\geq \mu_{0}>0 \,\,\,\text{and}\,\,\, \|\beta\|_{L^{\infty}(Q_{1}^{*})}\leq 1,
\end{equation}
where $\beta=(\beta_{1},\ldots,\beta_{n})$. Here $Q^{+}_{1}$ and $Q_{1}^{*}$ denote, respectively, the half and the thin parabolic cylinders of radius $1$ (see Section~\ref{Section2} for the precise definitions). The model \eqref{1.1} concerns a class of fully nonlinear parabolic problems endowed with oblique boundary conditions, a setting that has received little attention in the regularity theory literature.

In this context, a natural question concerns the continuity properties of solutions and the corresponding regularity theory. Focusing specifically on fully nonlinear parabolic equations, significant progress has been made in understanding the moduli of continuity of their solutions. Initially, we highlight the pioneering works of Krylov and Safonov \cite{KrySaf79,KrySaf80}, which initiated the regularity theory for viscosity solutions to fully nonlinear parabolic equations, where the authors considered non-divergence form parabolic equations. Two years later, Krylov \cite{Kry82,Kry83} studied elliptic and parabolic classes of fully nonlinear equations and developed, in particular, $C^{2+\alpha,\frac{2+\alpha}{2}}$ estimates for solutions to 
$$
u_{t}-F(D^{2}u)=0
$$
in the convex case; see also Goffi \cite{Gof24,Gof26} and Da Silva and Santos \cite{DaSSan24} for related results without the convexity assumption. 

Subsequently, Wang, in a series of four papers \cite{Wang,WangI,WangII,WangIII}, developed a regularity theory for viscosity solutions to fully nonlinear parabolic equations with variable coefficients. In particular, in \cite{WangI,WangII} he proved the estimates $C^{\alpha,\frac{\alpha}{2}}$, $C^{1+\alpha,\frac{1+\alpha}{2}}$, and $C^{2+\alpha,\frac{2+\alpha}{2}}$. For a non-exhaustive list of results in this direction, we refer the reader to \cite{AdiBanVer20,AraSaUrb24,CasPim17, CraKocSwi00,DaS19,DaSOch19,Kry18}. Finally, we highlight the seminal work of Da Silva and Teixeira \cite{DaST17} (see also \cite{AmaDos22} and \cite{ET14} for the elliptic scenario), who established a classification of the moduli of continuity of viscosity solutions for
\[
u_{t}-F(D^{2}u,x,t)=f(x,t)\,\,\, \text{in}\,\,\, Q_{1}.
\]
More specifically, under suitable integrability assumptions on the source term $f$ within the framework of anisotropic Lebesgue spaces, the authors obtained optimal regularity results in $C^{\alpha,\frac{\alpha}{2}}$, parabolic $C^{0,\mathrm{Log\text{-}Lip}}$, $C^{1+\alpha,\frac{1+\alpha}{2}}$, and parabolic $C^{1,\mathrm{Log\text{-}Lip}}$, by means of polynomial approximation arguments combined with geometric tangential methods.

On the other hand, when considering such parabolic models under an oblique boundary condition \eqref{1.1}, the literature becomes rather limited. This scarcity is due to the difficulty of addressing the problem in the interior of the domain simultaneously with the boundary condition, which is not merely a boundary datum as in problems with Dirichlet boundary conditions. Important developments have been obtained in this direction. For example, Lieberman \cite{Lie90} showed H\"{o}lder continuity of the gradient of solutions to the homogeneous version of problem \eqref{1.1}, assuming suitable regularity conditions on the operator $F$. Two years later, Nazarov and Ural'tseva \cite{NazUra92} established $C^{1,\alpha}$ regularity for solutions to the following quasilinear parabolic problem with an oblique boundary condition:
\begin{equation*}
\left\{
\begin{array}{rclcl}
\displaystyle u_{t} - \operatorname{tr}(A(Du,u,x,t)D^{2}u) &=& f(Du,u,x,t) & \text{in} & \Omega_\mathrm{T}=\Omega \times (0,\mathrm{T}), \\
\beta \cdot Du + u &=& g(x,t) & \text{on} & \mathrm{S}_{\mathrm{T}}=\partial \Omega\times (0,T), \\
u(x, 0) &=& u_{0}(x) & \text{on} & \Omega,
\end{array}
\right.
\end{equation*}
where $\mathrm{T}>0$ and $\Omega$ is a bounded domain of $\mathbb{R}^{n}$ with boundary $\partial \Omega\in C^{1,1}$.

In the same direction, Chatzigeorgiou and Milakis \cite{CM21} obtained a counterpart of the results of Wang cited above for viscosity solutions to \eqref{1.1} when the operator $F$ has constant coefficients. We also cite some related results in this parabolic setting with oblique boundary conditions, such as \cite{BSFR26,BHLp22,DonPan21,Lie96,ZhaZheZuo21}.

Regarding sharp estimates in the elliptic framework with oblique boundary conditions, Bessa, Da Silva, and Ricarte \cite{BSR26} investigated the classification of the modulus of continuity for viscosity solutions of
\begin{eqnarray}\label{eqelip}
\left\{
\begin{array}{rclcl}
F(D^{2}u,x)&=& f(x)& \mbox{in} &   \Omega,  \\
 \beta(x)\cdot Du+\gamma(x) u&=& g(x) &\mbox{on}& \partial \Omega,
\end{array}
\right.
\end{eqnarray}
where $F$ is a uniformly elliptic operator. Under suitable integrability assumptions in Lebesgue spaces, the authors derived optimal regularity results in $C^{0,\alpha}$, $C^{0,\mathrm{Log\text{-}Lip}}$, $C^{1,\alpha}$, and $C^{1,\mathrm{Log\text{-}Lip}}$ for viscosity solutions to \eqref{eqelip}.

Despite all the advances mentioned above, there remains a gap in the study of optimal regularity associated with the model \eqref{1.1}. With this in mind, the present work aims to fill this gap by investigating the boundary regularity of viscosity solutions to \eqref{1.1} under different assumptions on the source term $f$ and on the boundary data $\beta$ and $g$.

First, under the assumption of Hölder regularity of the boundary data $\beta$ and $g$, we establish Log-Lipschitz regularity (Theorem \ref{Theorem3.2}) and optimal $C^{1+\alpha, \frac{1+\alpha}{2}}$ regularity (Theorem \ref{Theorem4.2}) when the source term $f$ belongs to $L^{n+2}$ and $L^{p}$ ($n+2<p<\infty$), respectively. Then, under stronger regularity assumptions on the boundary data, namely $C^{1+\alpha,\frac{1+\alpha}{2}}$, we establish Schauder-type estimates (Theorem \ref{Theorem5.3}) when the source term enjoys Hölder regularity. Finally, under the same regularity assumptions on the boundary data, assuming that the source term is unbounded but satisfies a Bounded Mean Oscillation (BMO) condition (see Definition \ref{defc1loglip}), and relying on suitable estimates for the homogeneous problem with constant coefficients, we establish $C^{1,\mathrm{Log\text{-}Lip}}$ regularity (Theorem \ref{Theorem6.3}). Table \ref{tab:main-results} provides a concise summary of the main results of this work.
\begin{table}[h!]
\centering
\begin{tabular}{|c|c|c|}
\hline
\textrm{Source term $f$} &\textrm{Boundary data} & \textrm{Regularity of $u$} \\
\hline
$f \in L^{n+2}(Q^{+}_1)$ & $\beta,g\in C^{\alpha,\frac{\alpha}{2}}(\overline{Q_{1}^{*}})$&
$u \in C^{0,\mathrm{Log\text{-}Lip}}_{\mathrm{loc}}(\overline{Q^{+}_{1}})$ \\
\hline
$f \in L^{p}(Q^{+}_1)$,\ $p>n+2$ & $\beta,g\in C^{\alpha,\frac{\alpha}{2}}(\overline{Q_{1}^{*}})$&
$u \in C^{1,\min\{\alpha^{-},\,1-\frac{n+2}{p}\}}_{\mathrm{loc}}(\overline{Q^{+}_{1}})$ \\
\hline
$f \in C^{\alpha,\frac{\alpha}{2}}(\overline{Q_{1}^{+}})$ &$\beta,g\in C^{1+\alpha,\frac{1+\alpha}{2}}(\overline{Q_{1}^{*}})$&
$u \in C^{2+\alpha^{-},\frac{2+\alpha^{-}}{2}}_{\mathrm{loc}}(\overline{Q^{+}_{1}})$ \\
\hline
$f \in\mathrm{BMO}(Q^{+}_1)$ &$\beta,g\in C^{1+\alpha,\frac{1+\alpha}{2}}(\overline{Q_{1}^{*}})$&
$u \in C^{1,\mathrm{Log\text{-}Lip}}_{\mathrm{loc}}(\overline{Q^{+}_{1}})$ \\
\hline
\end{tabular}
\medskip
\caption{Summary of the main results}
\label{tab:main-results}
\end{table}

In view of the previously described setting, our results can be regarded as a generalization of \cite{CM21,DaST17}, a counterpart of \cite{BSR26}, and as being directly connected to \cite{AmaDos22,BRS,LiZha18,ET14} with respect to moduli of continuity.

\subsubsection*{Sharp Moduli of Continuity: Challenges and \textit{Strategiae}}

We conclude the introduction by presenting the key guiding ideas behind the study of universal moduli of continuity, the main difficulties that arise, and the strategy underlying our proofs.

Let \(u\) be a viscosity solution to \eqref{1.1}. Three aspects are of fundamental importance in determining the regularity of \(u\), namely: the integrability level of the forcing term \(f\), the regularity of the boundary data \(\beta\) and \(g\), and the regularity theory available for the associated homogeneous frozen-coefficient problem.

In light of this setting, and invoking the classical Dini–Campanato embedding (see \cite[Main Theorem]{Kov99}), our approach is to employ geometric tangential analysis to construct a polynomial that approximates the solution (cf. \cite{Caf89,WangII}), assuming small oscillation of the coefficients of the operator \(F\) and small \(L^{p}\)-norm of the source term \(f\). 

Afterward, using an iterative argument, we construct a sequence of polynomials that approximate \(u\) at a geometric rate, according to the modulus of continuity described in Table \ref{tab:poly-approx} below in each case. 
\begin{table}[h!]
\centering
\begin{tabular}{|c|c|c|c|}
\hline
\textrm{Regularity of $f$} &\textrm{Regularity of $\beta,g$} & \textrm{Polynomial}&\textrm{Geometric Rate}  \\
\hline
$L^{n+2}$ & $C^{\alpha,\frac{\alpha}{2}}$&
$a\cdot x+b$ &$r$\\
\hline
$L^{p}$,\ $p>n+2$ & $C^{\alpha,\frac{\alpha}{2}}$&
$a\cdot x+b$ &$r^{1+\min\{\alpha^{-},\,1-\frac{n+2}{p}\}}$\\
\hline
$C^{\alpha,\frac{\alpha}{2}}$ &$C^{1+\alpha,\frac{1+\alpha}{2}}$&
$\frac{1}{2}x^{T}\mathrm{M}x+a\cdot x+bt+c$& $r^{2+\gamma}$ ($0<\gamma<\alpha$) \\
\hline
$\mathrm{BMO}$ &$C^{1+\alpha,\frac{1+\alpha}{2}}$&
$\frac{1}{2}x^{T}\mathrm{M}x+a\cdot x+bt+c$& $r^{2}$\\
\hline
\end{tabular}
\caption{Polynomial approximation.}
\label{tab:poly-approx}
\end{table}

However, one of the main difficulties of our problem, in comparison with those discussed in the literature above, lies precisely in the equation satisfied by \(u\) on the boundary. To illustrate this point, given a quadratic polynomial \(P\), we note that the function \(v(x,t) := (u-P)(x,t)\) is a viscosity solution to

\begin{equation*}
\left\{
\begin{array}{rclcl}
F(D^2v+D^{2}P,x,t) - v_{t} &=& f(x,t)+(P)_{t}(x,t)& \text{in} & Q_{1}^{+}, \\
\beta(x,t) \cdot Dv &=& g(x,t)-\beta(x,t)\cdot DP(x,t) & \text{on} & Q_{1}^{*}.
\end{array}
\right.
\end{equation*}
Unlike the interior theory in \cite{ET14,DaST17} and the Dirichlet boundary setting in \cite{AmaDos22}, in addition to the perturbation affecting the equation in the interior, the model undergoes a considerably worse penalization on the flat boundary. In view of this challenging scenario, our strategy is to construct a sequence of polynomials that satisfies a suitable compatibility relation with the boundary data.

This parabolic approach, aimed at establishing optimal regularity, constitutes one of the main novelties of this work. Another key contribution is the analysis within the Lebesgue scale, which presents additional challenges due to the lack of a scaling structure naturally compatible with the underlying parabolic model.

Finally, it is important to stress that this approach is flexible enough to handle regularity at branching points for other parabolic problems without oblique boundary conditions, to obtain higher regularity results for problems with oblique derivatives on the boundary, and also to treat obstacle-type problems within the same general framework.

Beyond these theoretical extensions, it is worth noting that the interest in model \eqref{1.1} is also motivated by its natural applications. A prominent example arises in stochastic optimal control for diffusion processes with boundary reflection. In this setting, the optimal cost function solves a fully nonlinear parabolic Hamilton-Jacobi-Bellman (HJB) equation, while the reflection vector field at the boundary dictates the oblique derivative condition (see, for instance, \cite{FlemingSoner2006, Lions1985Neumann, LionsSznitman1984, LionsTrudinger1986}). Similar configurations also appear in geometric analysis, such as in mean curvature flows for dynamic capillary surfaces, where the oblique condition models the contact angle between a moving fluid interface and a solid boundary \cite{AltschulerWu1994, Huisken1989, Stahl1996}.
 
\subsubsection*{Outline of the paper} The remainder of the paper is organized as follows. Section \ref{Section2} is devoted to the assumptions and preliminary results needed for the main theorems. Next, in Section \ref{Section3}, we establish parabolic $C^{0,\mathrm{Log\text{-}Lip}}$ regularity for solutions when the source term belongs to the critical space $L^{n+2}$. In Section \ref{Section4}, we prove the optimal $C^{1+\alpha,\frac{1+\alpha}{2}}$ regularity. Section \ref{Section5} is devoted to the Schauder regularity result. Section \ref{Section6} is dedicated to deriving parabolic $C^{1,\mathrm{Log\text{-}Lip}}$ estimates for our model. Finally, in Section \ref{Section7} we present several remarks on the approach developed in this paper and examine possible extensions of our results to other functional spaces.

\section{Preliminaries}\label{Section2}
We begin by collecting some basic notation that will be used throughout the paper.
\begin{itemize}
\item[-] $B_{r}(x)\subset\mathbb{R}^{n}$ denotes the open ball of radius $r$ centered at $x$. In particular, when $x=0$ we simply write $B_r$;
\item[-] $B_r^+\coloneqq B_r\cap\mathbb{R}^n_+$, where $\mathbb{R}^n_+\coloneqq\{(x,y)\in\mathbb{R}^{n-1}\times\mathbb{R}:y>0\}$;
\item[-] $T_r(x)\coloneqq\{(z,x_n)\in\mathbb{R}^{n-1}\times\mathbb{R}:|z-x|<r\}=B_r(x)\cap\{y_n=x_n\}$ and, when $x=0$, we write $T_r=\{(z,0):|z|<r\}=B_r\cap\{y_n=0\}$;
\item[-] $Q_r(x,t)\coloneqq B_r(x)\times(t-r^2,t]$ denotes a parabolic cylinder. In particular, we write $Q_r$ when $(x,t)=(0,0)$;
\item[-] $Q_r^+(x,t)\coloneqq B_r^+(x)\times(t-r^2,t]$ and $Q_r^*(x,t)\coloneqq T_r(x)\times(t-r^2,t]$ denote, respectively, a half-cylinder and a thin-cylinder based on the flat boundary $\{y=0\}$;
\item[-] The \textit{parabolic distance} between points $P_1=(x,t)$ and $P_2=(y,s)$ in $\mathbb{R}^{n+1}$ is defined by
\begin{equation*}
d_{\text{par}}(P_1,P_2)\coloneqq \max\left\{|x-y|,\;|t-s|^{1/2}\right\}.
\end{equation*}
\item[-]  Given a function \( u = u(x,t) \), we denote its time derivative by \( u_{t} \), its spatial gradient by \( Du = (u_{x_{1}}, \ldots, u_{x_{n}}) \), and its spatial Hessian matrix by \( D^{2}u = (u_{x_{i}x_{j}})_{n \times n} \). 
\item[-]Given $u$ a integrable function on $\mathbb{R}^{n+1}$ we set for any $E\subset \mathbb{R}^{n+1}$ mensurable set
$$
(u)_E=\intav{E}u(x,t)\,dxdt.
$$
\end{itemize}

Next, throughout this paper, we will assume the following structural conditions on the
problem \eqref{1.1}:
\begin{enumerate}
\item [\bf(A1)]\label{HypA1} ({\bf Uniform parabolicity}) We assume that $F:\text{Sym}(n)\times Q^{+}_{1}\longrightarrow\mathbb{R}$ is a uniformly parabolic operator, i.e., there exist two constants $0<\lambda \le \Lambda$ such that
\begin{equation*}\label{unfpar}
\lambda\|\mathrm{Y}\|  \le F(\mathrm{X}+\mathrm{Y},x,t)-F(\mathrm{Y},x,t) \le \Lambda\|\mathrm{Y}\|
\end{equation*}
for all $(x,t)\in Q^{+}_{1}$ and $\mathrm{X}, \mathrm{Y} \in \textrm{Sym}(n)$ with $\mathrm{Y}\geq 0$. Moreover, for normalization purposes, we shall always assume that $F(0,x,t)=0$ for any $(x,t)\in Q^{+}_{1}$, which is not restrictive.

\item[\bf (A2)] ({\bf Regularity of the data}) \, The data satisfy $f \in C^0(Q^{+}_{1})\cap L^{p}(Q^{+}_{1})$ for $n+2\leq p<\infty $ $g \in C^{0}(Q_{1}^{*})$, and $\beta \in C^{0}(Q_{1}^{*}; \mathbb{R}^n)$ and satisfy  \eqref{1.2}.
\end{enumerate}

Next, we introduce the primordial notion of viscosity solutions to equation~\eqref{1.1} (cf. \cite{CM21}). First, in the interior of half-cylinder we have the following definition: 
\begin{definition}
Let $u\in C^{0}(\Omega)$ is a viscosity subsolution (resp. supersolution) to 
\begin{equation}\label{interiorequation}
F(D^{2}u,x,t)-u_{t}=f(x,t)\,\,\, \text{in}\,\,\, \Omega,    
\end{equation}
if, for every $\phi \in C^{2}(\Omega)$ that touches $u$ from above (resp. below) at a point $(x_0, t_0) \in \Omega$, we have that
$$
F(D^{2}\phi(x_{0},t_{0}),x_{0},t_{0})-\phi_{t}(x_{0},t_{0})\geq \, (\leq)\, f(x_{0},t_{0}).
$$
Moreover, we say that $u$ is a viscosity solution of \eqref{interiorequation} if it is both a subsolution and a supersolution.
\end{definition}
In the thin part of the cylinder we use the following definition.
\begin{definition}
We say that $\beta \cdot Du \geq (\leq)\, g$ on $Q_r^*$ in the viscosity sense if, for every point $(x_0,t_0) \in Q_r^*$ and every $C^{2}$ test function $\phi$ that touches $u$ from above (below) at $(x_{0},t_{0})$ in some half-cylinder $Q_\rho^+(x_{0},t_{0}) \subset Q_r^+$, one has
\[
\beta(x_{0},t_{0}) \cdot D\phi(x_{0},t_{0}) \geq (\leq)\, g(x_{0},t_{0}).
\]
If both inequalities hold, we say that $\beta \cdot Du = g$ on $Q_r^*$ in the viscosity sense.
\end{definition}

Regarding the coefficients of the operator $F$, we introduce the following definition.

\begin{definition}
Let $F:\mathrm{Sym}(n)\times Q_1^{+}\to\mathbb{R}$ be a parabolic operator. For a fixed point $(x_0,t_0)\in Q_1^{+}$, we define
\[
\Phi_F((x,t),(x_0,t_0))
:=
\sup_{\mathrm{M}\in \mathrm{Sym}(n)\setminus \{0\}}
\frac{\big|F(\mathrm{M},x,t)-F(\mathrm{M},x_0,t_0)\big|}{\|\mathrm{M}\|},
\]
which measures the oscillation of the coefficients of $F$ around $(x_0,t_0)$ (cf. \cite{CraKocSwi00}).
When $(x_0,t_0)=(0,0)$, we simply write $\Phi_F(x,t)$.

For the higher-order regularity estimates, we shall require the following definition of the coefficient oscillation around the point $(x_{0},t_{0})$, namely,
\[
\tilde{\Phi}_F((x,t),(x_0,t_0))
:=
\sup_{\mathrm{M}\in \mathrm{Sym}(n)}
\frac{\big|F(\mathrm{M},x,t)-F(\mathrm{M},x_0,t_0)\big|}{1+\|\mathrm{M}\|}.
\]
Again, when$(x_0,t_0)=(0,0)$, we simply write $\tilde{\Phi}_F(x,t)$.
\end{definition}
For what follows, we shall adopt the following notation:
$$
\mathcal{L}^{\pm}_{\lambda,\Lambda}[u]=:\mathcal{M}_{\lambda,\Lambda}^{\pm}(D^{2}u)-u_{t}.
$$
Here,$\mathcal{M}^{\pm}_{\lambda,\Lambda}$ are the Pucci extremal operators, defined by
\begin{equation*}
\mathcal{M}^{+}_{\lambda,\Lambda}(\mathrm{X})=\Lambda \sum_{e_{i}>0}e_{i}+\lambda \sum_{e_{i}<0}e_{i}\quad \text{and}\quad \mathcal{M}^{-}_{\lambda,\Lambda}(\mathrm{X})=\lambda \sum_{e_{i}>0}e_{i}+\Lambda \sum_{e_{i}<0}e_{i}
\end{equation*}
where $e_{i}=e_{i}(\mathrm{X})$ ($1\leq i\leq n$) are the eigenvalues of the matrix $\mathrm{X}$.
Now we can define the fundamental classes of parabolic equations (cf. \cite{CM21}).
\begin{definition}
We say that $u\in \underline{\mathcal{S}}_{p}(\lambda,\Lambda,f)$ in $\Omega$ if $u$ satisfies in the viscosity sense, 
$$
\mathcal{L}_{\lambda,\Lambda}^{+}[u]\leq f(x,t)\,\,\, \text{in}\,\,\, \Omega.
$$
Similarly, we say that $u\in \overline{\mathcal{S}}_{p}(\lambda,\Lambda,f)$ in $\Omega$ if, in the viscosity sense, $u$ satisfies
$$
\mathcal{L}_{\lambda,\Lambda}^{-}[u]\geq f(x,t)\,\,\, \text{in}\,\,\, \Omega.
$$
Moreover, we define $\mathcal{S}_{p}(\lambda,\Lambda,f)=\underline{\mathcal{S}}_{p}(\lambda,\Lambda,f)\cap \overline{\mathcal{S}}_{p}(\lambda,\Lambda,f)$.
\end{definition}

Later on, we will need Aleksandrov-Bakel'man-Pucci-Tso (A.B.P.T) type estimates, which is a direct consequence of  \cite[Theorem 2.5]{CM21}.
\begin{theorem}[{\bf A.B.P.T.-estimates}]\label{ABPTestimates}
Let $f\in C^{0}(\overline{Q_{r}^{+}})$, $f\in C^{0}(\overline{Q_{r}^{*}})$ and $u\in C^{0}(\overline{Q_{r}^{+}})$ satisfying 
\begin{equation*}
\left\{
\begin{array}{lclcl}
u\in \mathcal{S}_{p}(\lambda,\Lambda,f) & \text{in} & Q_{r}^{+}, \\
\beta \cdot Du + \gamma u = g(x,t) & \text{on} & Q_{r}^{*}.
\end{array}
\right.
\end{equation*}
Then, 
$$
\|u\|_{L^{\infty}(Q_{r}^{+})}\leq \|u\|_{L^{\infty}(\partial Q_{r}^{+}\setminus Q_{r}^{*})}+\mathrm{C}(\|f\|_{L^{n+1}(Q_{r}^{+})}+\|g\|_{L^{\infty}(Q_{r}^{*})}),
$$
for some universal constant.
\end{theorem}

\subsection{Some functional spaces}

Now we introduce the functional spaces that will be used throughout this manuscript.

\begin{definition}
Let $\alpha\in(0,1]$ and $Q$ be a domain of $\mathbb{R}^{n+1}$. We say that $u\in C^{\alpha,\frac{\alpha}{2}}(\overline{Q})$ if 

\[
\|u\|_{C^{\alpha,\frac{\alpha}{2}}(\overline{Q})}
:= \|u\|_{L^\infty(\overline{Q})} + [u]_{\alpha,\overline{Q}}<+\infty,
\]
where
\[
[u]_{\alpha,\overline{Q}}
:= \sup_{\substack{(x,t),(y,s)\in \overline{Q}\\(x,t)\neq (y,s)}}
\frac{|u(x,t)-u(y,s)|}{d_{par}((x,t),(y,s))^{\alpha}}.
\]
\end{definition}
\begin{remark}
We observe that if $u\in C^{\alpha,\frac{\alpha}{2}}(\overline{Q})$, then $u$ is $\alpha$-H\"older continuous with respect to the spatial variable $x$ and $\frac{\alpha}{2}$-H\"older continuous with respect to the temporal variable $t$.
\end{remark}
Next, we introduce the definition of the space of parabolic Log--Lipschitz continuous functions.
\begin{definition}
Let $Q$ be a domain of $\mathbb{R}^{n+1}$. We say a function $u$ is parabolically Log-Lipschitz continuous in $\overline{Q}$ if 
\[
[u]_{\text{par}-C^{0,\mathrm{Log\text{-}Lip}}(\overline{Q})}=\sup_{\substack{(x,t),(y,s)\in \overline{Q}\\(x,t)\neq (y,s)}}\frac{|u(x,t)-u(y,s)|}{d_{\text{par}}((x,t),(y,s))\left|\ln (d_{\text{par}}((x,t),(y,s)))\right|}<+\infty.
\]
Moreover, we denote by $\text{par}-C^{0,\mathrm{Log\text{-}Lip}}(\overline{Q})$, the space of all parabolically Log-Lipschitz functions in $Q$ and adopt the following norm
\[
\|u\|_{\text{par}-C^{0,\mathrm{Log\text{-}Lip}}(\overline{Q})}=\|u\|_{L^{\infty}(\overline{Q})}+[u]_{\text{par}-C^{0,\mathrm{Log\text{-}Lip}}(\overline{Q})}.
\]
\end{definition}
Now, for the next definition, given a function $u$ defined in a domain $Q\subset \mathbb{R}^{n+1}$, we set,
\[
[u]_{1+\alpha,\overline{Q}}=:\sup_{\substack{(x,t),(x,s)\in \overline{Q}\\t\neq s}}
\frac{|u(x,t)-u(x,s)|}{|t-s|^{\frac{1+\alpha}{2}}}.
\]
With this observation, we can define the Hölder spaces of the higher order.
\begin{definition}
We define the Hölder spaces $ C^{1+\alpha,\frac{1+\alpha}{2}}(\overline{Q})$ and $C^{2+\alpha,\frac{2+\alpha}{2}}(\overline{Q})$ for $\alpha\in(0,1]$, respectively, by set of all functions such that
\begin{eqnarray*}
\|u\|_{C^{1+\alpha,\frac{1+\alpha}{2}}(\overline{Q})} &=& \|u\|_{L^{\infty}(\overline{Q})} +\|Du\|_{L^{\infty}(\overline{Q})}+[Du]_{\alpha,\overline{Q}}+ [u]_{1+\alpha,\overline{Q}}  < +\infty, \\
\|u\|_{C^{2+\alpha,\frac{2+\alpha}{2}}(\overline{Q})} &=& \|u\|_{L^{\infty}(\overline{Q})} +\|Du\|_{L^{\infty}(\overline{Q})}+\|u_{t}\|_{L^{\infty}(\overline{Q})}+\|D^{2}u\|_{L^{\infty}(\overline{Q})}\\
&+&[u_{t}]_{\alpha,\overline{Q}}+[D^{2}u]_{\alpha,\overline{Q}}+ [Du]_{1+\alpha,\overline{Q}} < +\infty,
\end{eqnarray*}
respectively. 
\end{definition}

\begin{definition}\label{defc1loglip}
Let $Q$ be a domain of $\mathbb{R}^{n+1}$. We say a function $u$ is parabolically $C^{1,\mathrm{Log\text{-}Lip}}$ continuous in $\overline{Q}$ if 
\[
[u]_{\text{par}-C^{1,\mathrm{Log\text{-}Lip}}(\overline{Q})}=\sup_{\substack{(x,t),(y,s)\in \overline{Q}\\(x,t)\neq (y,s)}}\frac{\left|u(x,t)-[u(y,s)+D(y,s)\cdot (x-y)]\right|}{d_{\text{par}}^{2}((x,t),(y,s))\left|\ln (d_{\text{par}}((x,t),(y,s)))\right|}<+\infty.
\]
Moreover, we denote by $\text{par}-C^{1,\mathrm{Log\text{-}Lip}}(\overline{Q})$, the space of all parabolically Log-Lipschitz functions in $Q$ and adopt the following norm
\[
\|u\|_{\text{par}-C^{1,\mathrm{Log\text{-}Lip}}(\overline{Q})}=\|u\|_{L^{\infty}(\overline{Q})}+\|Du\|_{L^{\infty}(\overline{Q})}+[u]_{\text{par}-C^{1,\mathrm{Log\text{-}Lip}}(\overline{Q})}.
\]
\end{definition}

Another class of functions that we need to introduce is that of functions with \textit{Bounded Mean Oscillation}.
\begin{definition}\label{DefBOM}Let $Q\subset \mathbb{R}^{n+1}$ a Lipschitz domain, the space $\text{BMO}(Q)$ is defined as the set of all functions $f \in L^1_{loc}(Q)$ such that
$$[u]_{\text{BMO}(Q)} = \sup_{(x_0, t_0) \in Q, r > 0}  \intav{Q \cap Q_r(x_0, t_0)} |u(x, t) - u_{Q \cap Q_r(x_0, t_0)}| \, dx \, dt < \infty$$
\end{definition}

In our approach, we require a tool to approximate solutions of problem \eqref{1.1} by those of the homogeneous problem with constant coefficients. In summary, we have the following result.
\begin{lemma}[{\bf $F$-Caloric Approximation}] \label{Approx}
Consider $n+1 \le p < \infty$, and assume that conditions \(\mathrm{(A1)-(A2)}\) are satisfied. Given $\varepsilon > 0$, assume that $g \in C^{\alpha,\frac{\alpha}{2}}(\overline{Q^{*}_{1}})$ with $0 < \alpha < 1$ and $\|g\|_{C^{\alpha,\frac{\alpha}{2}}(\overline{Q^{*}_{1}})} \le \mathfrak{c}_{1}$ for some $\mathfrak{c}_{1} > 0$. Then, there exist positive constants $\delta = \delta(\varepsilon,n, \zeta_0, p, \lambda, \Lambda, \mathfrak{c}_{1}) < 1$ such that, if
\[
\max\left\{ \left(\intav{Q_{1}^{+}}\Phi_{F}(x,t)^{p}dxdt\right)^{\frac{1}{p}},\,\|f\|_{L^{p}(Q^{+}_{1})}  \right\} \le \delta,
\]
then any two viscosity solutions $u$ (normalized\footnote{A normalized solution means that $\|u\|_{L^{\infty}(Q^{+}_1)}\leq 1$.}) and $\mathrm{v}$ to
\[
\left\{
\begin{array}{rclcl}
F(D^2u,x,t) - u_{t} &=& f(x,t) & \text{in} & Q^{+}_{1}, \\
\beta \cdot Du&=& g(x,t) & \text{on} & Q_{1}^{*},
\end{array}
\right.
\]
and
\[
\left\{
\begin{array}{rclcl}
F(D^2 \mathrm{v},0,0) - \mathrm{v}_{t} &=& 0 & \text{in} & Q^{+}_{\frac{3}{4}}, \\
\beta \cdot D\mathrm{v}&=& g(x,t) & \text{on} & Q_{\frac{3}{4}}^{*}, \\
\mathrm{v} &=& u & \text{on} & \partial_{p}Q^{+}_{\frac{3}{4}}\setminus Q_{\frac{3}{4}}^{*}
\end{array}
\right.
\]
satisfy the following estimate
\[
\|u - \mathrm{v}\|_{L^{\infty}\left(Q^{+}_{\frac{3}{4}}\right)} \le \varepsilon.
\]
\end{lemma}

\begin{proof}
The proof follows along the same lines as in \cite[Lemma 5.3]{BHLp22} and \cite[Lemma 3.3]{BSFR26}, so we omit it here.
\end{proof}
\begin{remark}
The previous approximation result remains valid if the smallness assumption on the oscillation of the coefficients is formulated in terms of $\tilde{\Phi}_{F}$ instead.
\end{remark}
In the context of the parabolic $C^{1,\mathrm{Log\mbox{-}Lip}}$ estimates, we shall require the following approximation result in the BMO setting.
\begin{lemma}[{\bf $F$-Caloric Approximation-BMO}] \label{ApproxBMO}
Let $n+1\leq p<\infty$, and assume that conditions \(\mathrm{(A1)-(A2)}\) are satisfied. Given $\varepsilon > 0$, assume that $g \in C^{\alpha,\frac{\alpha}{2}}(\overline{Q^{*}_{1}})$ with $0 < \alpha < 1$ and $\|g\|_{C^{\alpha,\frac{\alpha}{2}}(\overline{Q^{*}_{1}})} \le \mathfrak{c}_{1}$ for some $\mathfrak{c}_{1} > 0$. Then, there exist positive constants $\delta = \delta(\varepsilon,n,p, \lambda, \Lambda, \mathfrak{c}_{1}) < 1$ such that, if
\[
\max\left\{ \left(\intav{Q_{1}^{+}}\tilde{\Phi}_{F}(x,t)^{p}dxdt\right)^{\frac{1}{p}},\,[f]_{\mathrm{BMO}(Q^{+}_{1})}  \right\} \le \delta,
\]
then any two viscosity solutions $u$ (normalized) and $\mathrm{v}$ to
\begin{equation}\label{parabolic problem}
    \left\{
\begin{array}{rclcl}
F(D^2u,x,t) - u_{t} &=& f(x,t) & \text{in} & Q^{+}_{1}, \\
\beta \cdot Du&=& g(x,t) & \text{on} & Q_{1}^{*},
\end{array}
\right.
\end{equation}
and
\begin{equation}\label{probBMO}
\left\{
\begin{array}{rclcl}
F(D^2 \mathrm{v},0,0) - \mathrm{v}_{t} &=& (f)_{Q_1^+} & \text{in} & Q^{+}_{\frac{3}{4}}, \\
\beta \cdot D\mathrm{v}&=& g(x,t) & \text{on} & Q_{\frac{3}{4}}^{*}, \\
\mathrm{v} &=& u & \text{on} & \partial_{p}Q^{+}_{\frac{3}{4}}\setminus Q_{\frac{3}{4}}^{*}
\end{array}
\right.
\end{equation}
satisfy the following estimate
\[
\|u - \mathrm{v}\|_{L^{\infty}\left(Q^{+}_{\frac{3}{4}}\right)} \le \varepsilon.
\]
\end{lemma}
\begin{proof}
Let $u$ be a normalized viscosity solution of \eqref{parabolic problem} where $f \in \mathrm{BMO}(Q_{1}^{+})$ with $[f]_{\mathrm{BMO}(Q_{1}^{+})} \le \tilde{\delta}$ (where $\tilde{\delta} > 0$ will be chosen later) .

Let $(f)_{Q_1^+}$ be the integral average of $f$ over the cylinder $Q_{1}^{+}$ and define $\tilde{F}: \text{Sym}(n) \times Q_1^+ \to \mathbb{R}$ and the  $\tilde{f}: Q_{1}^{+} \to \mathbb{R}$ as follows:
$$\tilde{F}(M, x, t) := F(M, x, t) - (f)_{Q_1^+} \ \ \ \ \text{and} \ \ \ \ \tilde{f}(x,t) := f(x,t) - (f)_{Q_1^+}.$$
It is easy to see that $u$ still a viscosity solution of 
$$\tilde{F}(D^{2}u, x, t) - u_t = \tilde{f}(x,t) \ \ \ \ \text{in} \ \ \ \ Q^+_1$$
with same boundary condition of \eqref{parabolic problem}. Moreover, by the definition of $\tilde{F}$, we immediately obtain
\[
\tilde{\Phi}_{\tilde{F}}=\tilde{\Phi}_{F}.
\]
Consequently, for $p>n+1$
$$
\left(\intav{Q_{1}^{+}}\tilde{\Phi}_{\tilde{F}}(x,t)^{p}dxdt\right)^{\frac{1}{p}}=\left(\intav{Q_{1}^{+}}\tilde{\Phi}_{F}(x,t)^{p}dxdt\right)^{\frac{1}{p}}\leq \delta.
$$
Furthermore, by John-Nirenberg inequality,
$$\|\tilde{f}\|_{L^p(Q_1^+)} = \left( \int_{Q_1^+} |f(x,t) - (f)_{Q_1^+}|^p \,dx\,dt \right)^{\frac{1}{p}} \le |Q_1^+|^{\frac{1}{p}}\mathrm{C}(n,p) [f]_{\mathrm{BMO}(Q_1^+)}.$$
To ensure that $\|\tilde{f}\|_{L^p(Q_1^+)} \le \delta$, it suffices to require in our initial hypothesis that the BMO seminorm is suitably small, specifically $[f]_{\mathrm{BMO}(Q_1^+)} \le \tilde{\delta} := \frac{\delta}{|Q_1^+|^{\frac{1}{p}}\mathrm{C}(n,p)}$.

Since both hypotheses are perfectly satisfied (i.e., $\max[\|\Phi_{\tilde{F}}\|_{L^p}, \|\tilde{f}\|_{L^p}] \le \delta$), we apply Lemma \ref{Approx} to obtain that existence of a function $v$ such that 
$$
\sup_{Q_{3/4}^+}|u - v| \le \varepsilon,$$
where $\mathrm{v}$ solves the limit equation with the coefficients frozen at the origin:
\[
\left\{
\begin{array}{rclcl}
F(D^2 \mathrm{v},0,0) - \mathrm{v}_{t} &=& 0 & \text{in} & Q^{+}_{\frac{3}{4}}, \\
\beta \cdot D\mathrm{v}&=& g(x,t) & \text{on} & Q_{\frac{3}{4}}^{*}, \\
\mathrm{v} &=& u & \text{on} & \partial_{p}Q^{+}_{\frac{3}{4}}\setminus Q_{\frac{3}{4}}^{*}.
\end{array}
\right.
\]
By the definition of $\tilde{F}$, it follows that $\mathrm{v}$ is a viscosity solution of \eqref{probBMO}. This completes the proof.
\end{proof}


\section{The first borderline case: Log-Lipschitz regularity}\label{Section3}

In this section we present our first main result, namely the Log--Lipschitz regularity of solutions to problem \eqref{1.1}.  Under the assumption that the boundary data $\beta$ and $g$ are H\"older continuous, and that the source term $f$ belongs to the borderline Lebesgue space $L^{n+2}$, we obtain the optimal modulus of continuity. The key ingredient is the following lemma, which, under suitable smallness conditions, provides an approximation of normalized solutions to problem \eqref{1.1} by affine profiles with decay rate of order $\rho^{1}$.

\begin{lemma}\label{Lemma3.1}
Let $u$ be a normalized solution to \eqref{1.1} with $\beta,g\in C^{\alpha,\frac{\alpha}{2}}(\overline{{{Q}_{1}^{*}}})$. There exists a universal constant $\delta>0$ and a radius $0<\rho\leq e^{-1}$ such that, if 
\[
\max\left\{ \left(\intav{Q_{1}^{+}}\Phi_{F}(x,t)^{n+2}dxdt\right)^{\frac{1}{n+2}},\,\|f\|_{L^{n+2}(Q^{+}_{1})}  \right\} \le \delta,
\]
then we can find an affine profile $\ell(x)=a\cdot x+b$ such that:
\begin{itemize}
\item[(i)] $\displaystyle\sup_{\overline{Q^{+}_{\rho}}}|u-\ell|\leq \rho$;
\item[(ii)] $\beta(0,0)\cdot a=g(0,0)$;
\item[(iii)] $|a|+|b|\leq \mathfrak{C}$, for some universal constant $\mathfrak{C}>0$. 
\end{itemize}
\end{lemma}

\begin{proof}
Fix $\varepsilon>0$, to be determined later. By applying Lemma~\ref{Approx}, we obtain the existence of a function $v$ solving
\[
\left\{
\begin{array}{rclcl}
F(D^2 \mathrm{v},0,0) - \mathrm{v}_{t} &=& 0 & \text{in} & Q^{+}_{\frac{3}{4}}, \\
\beta \cdot D\mathrm{v}&=& g(x,t) & \text{on} & Q_{\frac{3}{4}}^{*}, \\
\mathrm{v} &=& u & \text{on} & \partial_{p}Q^{+}_{\frac{3}{4}}\setminus Q_{\frac{3}{4}}^{*}
\end{array}
\right.
\]
and satisfy 
\begin{equation}\label{3.1}
\|u - \mathrm{v}\|_{L^{\infty}\left(Q^{+}_{\frac{3}{4}}\right)} \le \varepsilon.
\end{equation}
By the structural assumptions on the governing operator $F$ and on the boundary data $\beta$ and $g$, the available regularity theory yields that $\mathrm{v}\in C^{1+\alpha,\frac{1+\alpha}{2}}\big(\overline{Q^{+}_{\frac{2}{3}}}\big)$ (cf.~\cite[Theorem 4.10]{CM21}), together with the estimate
\begin{equation}\label{3.2}
\|\mathrm{v}\|_{C^{1+\alpha,\frac{1+\alpha}{2}}(\overline{Q^{+}_{\frac{2}{3}}})}\leq \mathrm{C}(\mu_{0},n,\alpha,\lambda,\Lambda, \|\beta\|_{C^{\alpha,\frac{\alpha}{2}}(\overline{Q^{*}_{\frac{2}{3}}})},\|g\|_{C^{\alpha,\frac{\alpha}{2}}(\overline{Q^{*}_{\frac{2}{3}}})}).
\end{equation}
Then the affine profile $\ell(x)=a\cdot x+b$, with $a=D\mathrm{v}(0,0)$ and $b=\mathrm{v}(0,0)$, is well defined. Moreover, for every $(x,t)\in \overline{Q^{+}_{\frac{2}{3}}}$ one has
\begin{eqnarray*}
|\mathrm{v}(x,t)-\ell(x)|&\leq&|\mathrm{v}(x,t)-\mathrm{v}(0,t)-D\mathrm{v}(0,t)\cdot x|+|\mathrm{v}(0,t)-\mathrm{v}(0,0)|+|D\mathrm{v}(0,t)-D\mathrm{v}(0,0)||x|\nonumber\\
&\leq&[D\mathrm{v}]_{\alpha,\overline{Q^{+}_{\frac{2}{3}}}}|x|^{1+\alpha}+[\mathrm{v}]_{1+\alpha,\overline{Q^{+}_{\frac{2}{3}}}}|t|^{\frac{1+\alpha}{2}}+[D\mathrm{v}]_{\alpha,\overline{Q^{+}_{\frac{2}{3}}}}|x||t|^{\frac{\alpha}{2}}\nonumber\\
&\stackrel{\eqref{3.2}}{\leq}&\mathrm{C}(|x|^{1+\alpha}+|t|^{\frac{1+\alpha}{2}}+|x||t|^{\frac{\alpha}{2}})\nonumber\\
&\leq&\mathrm{C}_{0}d_{\text{par}}((x,t),(0,0))^{1+\alpha}.
\end{eqnarray*}
Hence,
\begin{equation}\label{3.3}
\sup_{Q^{+}_{r}}|\mathrm{v}-\ell|\leq \mathrm{C}_{0}r^{1+\alpha},\,\, \forall r\in\left(0,\frac{2}{3}\right).
\end{equation}

Now, we can define the constants $\varepsilon$ and $\rho$ as follows
\[
\rho=\min\left\{e^{-1},\left(\frac{1}{2\mathrm{C}_{0}}\right)^{\frac{1}{\alpha}}\right\}\,\,\, \text{and}\,\,\, \varepsilon=\frac{1}{2}\rho.
\]
With this choice, the constant $\delta>0$ provided by Lemma~\ref{Approx} is well defined. Moreover, by construction, the affine function $\ell$ immediately satisfies conditions (ii) and (iii). Finally, using \eqref{3.1} and \eqref{3.3}, we obtain
\begin{equation*}
\sup_{\overline{Q_{\rho}^{+}}}|u-\ell|\leq \sup_{Q^{+}_{\frac{3}{4}}}|u-v|+\sup_{\overline{Q^{+}_{\rho}}}|v-\ell|\leq \varepsilon+\mathrm{C}_{0}\rho^{1+\alpha}=\rho,
\end{equation*}
which completes the proof.
\end{proof}

We are now in a position to prove our first main result.

\begin{theorem}[\bf Log-Lipschitz regularity]\label{Theorem3.2}
Let $u$ be a viscosity solution to \eqref{1.1}, where $\beta,g\in C^{\alpha,\frac{\alpha}{2}}(\overline{Q^{*}_{1}})$ and $f\in L^{n+2}(Q^{+}_{1})\cap C^{0}(Q^{+}_{1})$. There exists a constant $\zeta_0>0$ and $r_{0}>0$, which depends only on $n$, $p$, $\lambda$, $\Lambda$, $\mu_{0}$, $\alpha$, and $\|\beta\|_{C^{\alpha,\frac{\alpha}{2}}(\overline{Q_{1}^{*}})}$  such that, if
\begin{equation*}
\sup_{(x,t)\in Q^+_{1/2}}\left(\intav{Q_{r}(x_,t)\cap Q_{1}^{+}}\Phi_{F}((y,s);(x,t))^{n+2}dyds\right)^{\frac{1}{n+2}}\leq \zeta_{0},
\end{equation*}
holds for all $r\leq r_0$, then $u\in {\text{par}-C^{0,\mathrm{Log\text{-}Lip}}}\left(\overline{Q^+_{\frac{1}{2}}}\right)$ and the following estimate holds 
\begin{equation}\label{3.4}
\|u\|_{\text{par-}C^{0,\mathrm{Log\text{-}Lip}}(\overline{Q^{+}_{\frac{1}{2}})}}\leq\mathrm{C}\Bigl(\|u\|_{L^{\infty}(Q^{+}_{1})}+\|f\|_{L^{n+2}(Q^{+}_{1})}+\|g\|_{C^{\alpha,\frac{\alpha}{2}}(\overline{Q^{*}_{1}})}\Bigr),
\end{equation}
where $\mathrm{C}$ is a positive constant depending only on $n$, $\lambda$, $\Lambda$, $\mu_{0}$, $\alpha$, and $\|\beta\|_{C^{\alpha,\frac{\alpha}{2}}(\overline{Q_{1}^{*}})}$.
\end{theorem}

\begin{proof}
First, by the normalization argument, we can suppose without loss of generality that
\begin{equation}
    \|u\|_{L^{\infty}(Q^{+}_{1})}\leq 1,\quad  \|f \|_{L^{n+2}(Q^{+}_{1})}\leq \delta,\quad  \|g\|_{C^{\alpha,\frac{\alpha}{2}}(\overline {Q^{*}_{1}})}\leq 1 \quad \text{and} \quad g(0,0)=0,
\end{equation}
where $\delta>0$ is the constant from Lemma \ref{Lemma3.1}.

In this case, we first establish regularity at points in the closure of the thin cylinder $Q^*_{\frac{1}{2}}$. By a standard localization and translation argument, it suffices to verify this regularity at $(0,0)\in Q^*_{\frac{1}{2}}$.

Assuming that the problem satisfies these conditions, let $\zeta_{0}=\delta$ and $r_0=\rho$ be as in Lemma \ref{Lemma3.1}. Our goal is to construct a sequence of affine functions $(\ell_k)_{k\geq 0},~\ell_k(x)=a_k\cdot x+b_k, $ satisfying:
\begin{itemize}
    \item[(i)] $\displaystyle\sup_{\overline{Q^+_{\rho^k}}}|u-\ell_k|\leq\rho^k$;
    \item[(ii)] $\rho^{k}|a_{k+1}-a_k|+|b_{k+1}-b_k|\leq \mathfrak{C}\rho^{k}$;
    \item[(iii)] $\beta(0,0)\cdot a_k=0$,
\end{itemize}
where $\rho>0$ denotes the radius of the half-cylinder introduced in Lemma \ref{Lemma3.1} and $\mathfrak{C}>0$ is the universal boundedness constant from the same lemma. Indeed, we proceed by induction on $k$. 

For $k=0$, consider the affine function $\ell_0$ with $a_0=0$ and $b_0=0$ satisfying
\begin{equation*}
    \sup_{Q^+_{\rho^0}}|u-\ell_0| =\|u\|_{Q^*_1}\leq1=\rho^0~\mbox{ and } ~\beta(0,0)\cdot a_0=0.
\end{equation*}
Now, assuming the existence of $\ell_0, ...,\ell_k$ satisfying (i)-(iii), consider the auxiliary function 
\begin{equation*}
    {u}_k(x)=\dfrac{u(\rho^k x,\rho^{2k}t)-\ell_k(\rho^k x)}{\rho^k},~~ (x,t)\in Q^+_1\cup Q^*_1,
\end{equation*}
in order to obtain the affine function $\ell_{k+1}$. It is possible to check that ${u}_k$ is a viscosity solution to
$$
\left\{
\begin{array}{rclcl}
{F}_k(D^2 {u}_k,x,t) - ({u}_k)_{t} &=& {f}_k(x,t) & \text{in} & Q^{+}_{1}, \\
{\beta}_k \cdot D{u}_k&=& {g}_k(x,t) & \text{on} & Q_{1}^{*},
\end{array}
\right.
$$
where
$$
\left\{
\begin{array}{rcl}
{F}_k(X,x,t) & \defeq & \rho^{k}F\left(\frac{1}{\rho^{k}} X,\rho^{k}x,\rho^{2k}t \right) \\
{f}_k(x,t)&\defeq&\rho^{k}f(\rho^{k}x,\rho^{2k}t)\\
{\beta}_k(x,t) & \defeq & \beta(\rho^{k}x,\rho^{2k}t)\\
{g}_k(x,t)& \defeq & g(\rho^{k}x,\rho^{2k}t)-\beta(\rho^{k}x,\rho^{2k}t)\cdot a_{k}.
\end{array}
\right.
$$
Now, we note that ${F}_k$ is a $(\lambda,\Lambda)-$elliptic operator and 
\[
\left(\intav{Q_{1}^{+}}|\Phi_{{F}_k}(x,t)|^{n+2}dxdt\right)^{\frac{1}{n+2}}=\left(\intav{Q_{\rho^{k}}^{+}}|\Phi_{F}(x,t)|^{n+2}dxdt\right)^{\frac{1}{n+2}}\leq \delta.
\]
Additionally, by the induction hypothesis, we also see that
\begin{align*}
    & \|{u}_k\|_{L^{\infty}(Q^{+}_{1})} =\dfrac{1}{\rho^k}\sup_{Q^+_{\rho^k}}|u-\ell_k| \leq  1
    \end{align*}
    and
    \begin{align*}
    & \|f_k \|_{L^{n+2}(Q^{+}_{1})}= \rho^{k}\left(\int_{Q_{1}^{+}}|f(\rho^{k}x,\rho^{2k}t)|^{n+2}\,dxdt\right)^{\frac{1}{n+2}}=\|{f} \|_{L^{n+2}(Q^{+}_{\rho^k})}\leq \|{f} \|_{L^{n+2}(Q^{+}_{1})} \leq 
    \delta.   
\end{align*}
Now, turning to the boundary data, we can verify that
\begin{eqnarray*}
    \|\beta_k\|_{C^{\alpha,\frac{\alpha}{2}}(\overline{Q_1^*})}
 &= & \|\beta_k \|_{L^\infty(\overline{Q_1^+})} +  \sup_{\substack{(x,t),(y,s)\in \overline{Q_{1}^*}\\(x,t)\neq (y,s)}}
\dfrac{|\beta(\rho^k x,\rho^{2k}t)-\beta(\rho^{k}y,\rho^{2k}s)|}{d_{\text{par}}((x,t),(y,s))^{\alpha}} \\
&\leq & \|\beta \|_{L^\infty(\overline{Q_1^*})} +  \sup_{\substack{(x,t),(y,s)\in \overline{Q_{\rho^{k}}^*}\\(x,t)\neq (y,s)}}
\dfrac{|\beta( x,t)-\beta(y,s)|}{\max\left\{\left|\frac{x}{\rho{k}}-\frac{y}{\rho^{k}}\right|,\left|\frac{t}{\rho{2k}}-\frac{s}{\rho^{2k}}\right|^\frac{1}{2}\right\}^{\alpha}} \\
&=& \|\beta\|_{L^\infty(\overline{Q_1^*})} +\rho^k [\beta]_{\alpha,\overline{Q^*_{\rho^k}}} \\
&\leq&\|\beta\|_{L^\infty(\overline{Q_1^*})} +\rho^k [\beta]_{\alpha,\overline{Q^*_{1}}}< \infty
\end{eqnarray*}
Next, for the boundary source term $g_{k}$, observe that since $g(0,0)=0$ and, by the induction hypothesis, $\beta(0,0)\cdot a_{k}=0$, it follows that $g_{k}(0,0)=0$. Hence,
\[
\|g_{k}\|_{L^{\infty}(\overline{Q_{1}^{*}})}\leq[g_{k}]_{\alpha,\frac{\alpha}{2},\overline{Q_{1}^{*}}}.
\]
Therefore, in analogy with the estimate for the boundary term $\beta_{k}$, we obtain that\begin{eqnarray*}
   \|g_k\|_{C^{\alpha,\frac{\alpha}{2}}(\overline{Q_1^*})}
&=& \|g_k\|_{L^\infty(\overline{Q_1^*})} +  \sup_{\substack{(x,t),(y,s)\in \overline{Q_{1}^*}\\(x,t)\neq (y,s)}}
\frac{|g_k(x,t)-g_k(y,s)|}{d_{par}((x,t),(y,s))^{\alpha}} \\
&\leq& \|g\|_{L^\infty(\overline{Q_1^*})} + |a_k| \|\beta\|_{L^\infty(\overline{Q_1^*})} + \rho^k [g]_{\alpha,\overline{Q^*_{\rho^k}}}+\rho^k |a_k|[\beta]_{\alpha,\overline{Q^*_{\rho^k}}} <\infty.
\end{eqnarray*}
Then, we fall into the hypotheses of Lemma \ref{Lemma3.1}. Thus, we can find an $\tilde{\ell}(x)=\tilde{a}\cdot x+\tilde{b}$ affine function such that 
\begin{equation}
    \sup_{Q^+_{\rho}} |{u}_k-\tilde{\ell}|\leq \rho,
\end{equation}
where $|\tilde{a}|+|\tilde{b}|\leq \mathfrak{C}$ and ${\beta}_k(0,0)\cdot \tilde{a}={g}_k(0,0)=0.$

Now, by defining $\ell_{k+1}=a_{k+1}\cdot x+b_{k+1}$, where $a_{k+1}=a_k+\tilde{a}$ and $b_{k+1}=b_k+\rho^k\tilde{b}$, we verify that
\begin{itemize}
    \item[(a)] $ \displaystyle\sup_{Q^+_{\rho^{k+1}}}|u-\ell_{k+1}|=\rho^k\sup_{Q^+_{\rho}}|{u}_k-\tilde{\ell}| \leq\rho^{k+1}$,
    \item[(b)] $\rho^{k}|a_{k+1}-a_k|+|b_{k+1}-b_{k}|=\rho^{k}(|\tilde{a}|+|\tilde{b}|)\leq \mathfrak{C}\rho^{k}$,
    \item[(c)] $\beta(0,0)\cdot a_{k+1}={\beta}_k(0,0)\cdot a_{k+1}=\beta(0,0)\cdot a_k+\rho^k{\beta}_k(0,0)\cdot\tilde{a}=0$.
\end{itemize}
We conclude that there exists a sequence of affine functions satisfying (i)-(iii). Now, observe the following growth controls for the sequences
\begin{equation}
|a_k|\leq\sum_{i=0}^{k-1}|a_{i+1}-a_{i}|\leq \mathfrak{C}k
\end{equation}
and by hypothesis (i) it follows that  
\begin{equation}
|u(0,0)-b_k|=|u(0,0)-\ell_k(0)|\leq \displaystyle\sup_{\overline{Q^+_{\rho^{k}}}}|u-\ell_{k}|\leq\rho^k.
\end{equation}
In particular, $u(0,0)=\displaystyle\lim_{k\to\infty}b_k$. Before estimating $[u]_{\text{par}-C^{0,\mathrm{Log\text{-}Lip}}(\overline{Q^+_{1/2}})}$, note that for every $r\in(0,\rho)$ we can find $k\in\mathbb{N}$ such that $\rho^{k+1}\leq r<\rho^k$, with $k\geq 1$. Thus,
\begin{eqnarray*}
\displaystyle\sup_{(x,t)\in Q^+_r}|u(x,t)-u(0,0)| &=&  \displaystyle\sup_{(x,t)\in Q^+_r}|u(x,t)-\ell_k(x)+\ell_k(x)-u(0,0)|   \\
& \leq & \sup_{(x,t)\in Q^+_r}|u(x,t)-\ell_k(x)| +|a_k||x|+ |b_k-u(0,0)| \\
& \leq & \rho^k+k\mathfrak{C}r+\rho^k  \\
& < &k\mathfrak{C}\rho^k+2\rho^k =\rho^{k+1}\left(\dfrac{\mathfrak{C}k}{\rho}+\dfrac{2}{\rho}\right)  \\
& < & rk\mathrm{C}_1,~~(\mathrm{C}_1>0 \mbox{ universal}) \\
&\leq& r\dfrac{-\log r}{-\log\rho} \mathrm{C}_1  \\
& \leq & \mathrm{C}_1r\log (r^{-1}).
\end{eqnarray*}
This shows the Log-Lipschitz continuity at the origin. By the initial observation and continuity of $u$, we conclude that $u$ is Log-Lipschitz continuous on $\overline{Q_{\frac{1}{2}}^{*}}$.

Finally, we prove that $u\in \text{par}-C^{0,\mathrm{Log\text{-}Lip}}(\overline{Q^+_{\frac{1}{2}}})$. In effect, given $(x,t)\in Q_{\frac{1}{2}}^{+}$ and $(y,s)\in Q^{*}_{\frac{1}{2}}$, we have two cases to analyze:
\begin{itemize}
    \item[Case 1:] $d_{\text{par}}\left((x,t),(y,s)\right)<\rho$ \\
    For $r=d_{\text{par}}\left((x,t),(y,s)\right)$, we have
    \begin{equation*}
        \dfrac{|u(x,t)-u(y,s)|}{d_{\text{par}}\left((x,t),(y,s)\right)\log (d_{\text{par}}\left((x,t),(y,s)\right)^{-1})} \leq 
        \dfrac{\mathrm{C}_1 r\log(r^{-1})}{r\log(r^{-1})}=\mathrm{C}_1.
    \end{equation*}

     \item[Case 2:] $\rho\leq d_{\text{par}}\left((x,t),(y,s)\right)\leq\dfrac{1}{e}$ \\
    For $r=d_{\text{par}}\left((x,t),(y,s)\right)$, we have $\log(r^{-1})\geq\log e=1$ and
    \begin{equation*}
        \dfrac{|u(x,t)-u(y,s)|}{d_{\text{par}}\left((x,t),(y,s)\right)\log (d_{\text{par}}\left((x,t),(y,s)\right)^{-1})} \leq 
        \dfrac{2\|u\|_{L^\infty(\overline{Q^+_{1/2}})}}{\rho}\leq \dfrac{2}{\rho}\defeq \mathrm{C}_2.
    \end{equation*}
 \end{itemize}
Using $\mathrm{C}\defeq\max\{\mathrm{C}_1,\mathrm{C}_2\}$ we obtain
\begin{equation*}
[u]_{\text{par}-C^{0,\mathrm{Log\text{-}Lip}}(\overline{Q^+_{\frac{1}{2}}})}=\displaystyle\sup_{Q^+_{\frac{1}{2}}} \dfrac{|u(x,t)-u(y,s)|}{d_{\text{par}}\left((x,t),(y,s)\right)\log (d_{\text{par}}\left((x,t),(y,s)\right)^{-1})} \leq \mathrm{C}. 
\end{equation*}
Combining with interior estimates \cite[Theorem 4.2]{DaST17}, the estimate \eqref{3.4} follows. This ends the proof.
\end{proof}

As a byproduct of the Log-Lipschitz regularity estimates, we obtain $C^{\alpha,\frac{\alpha}{2}}$ regularity for every $\alpha\in (0,1)$, as summarized in the following result.

\begin{corollary}\label{Corollary3.3}
Assume that the hypotheses of Theorem \ref{Theorem3.2} are satisfied. Then every viscosity solution to \eqref{1.1} belongs to $C^{\alpha,\frac{\alpha}{2}}(\overline{Q_{\frac12}^{+}})$ for every $\alpha\in(0,1)$.
\end{corollary}
\begin{proof}
By Theorem \ref{Theorem3.2}, for any $\alpha\in(0,1)$ and
$X=(x,t),Y=(y,s)\in\overline{Q_{\frac{1}{2}}^{+}}$, we have
\begin{eqnarray*}
\frac{|u(X)-u(Y)|}{d_{\text{par}}(X,Y)^{\alpha}}&=& \frac{|u(X)-u(Y)|}{d_{\text{par}}(X,Y)\log(d_{\text{par}}(X,Y)^{-1})} d_{\text{par}}(X,Y)^{1-\alpha}\log(d_{\text{par}}(X,Y)^{-1})\\
&\leq&[u]_{\text{par}-C^{0,\mathrm{Log\text{-}Lip}}(\overline{Q_{\frac{1}{2}}^{+}})}d_{\text{par}}(X,Y)^{1-\alpha}\log(d_{\text{par}}(X,Y)^{-1})\\
&\leq& \mathrm{C}_{\alpha} [u]_{\text{par}-C^{0,\mathrm{Log\text{-}Lip}}(\overline{Q_{\frac{1}{2}}^{+}})},
\end{eqnarray*}
where the last inequality follows from the boundedness of the function $t\longmapsto t^{1-\alpha}\log t^{-1}$ in $(0,1)$. This completes the proof.
\end{proof}


\section{Optimal Hölder Regularity of the gradient}\label{Section4}

In this part, we establish the sharp asymptotic $C^{1+\alpha,\frac{1+\alpha}{2}}$ regularity for viscosity solutions to the model \eqref{1.1} when the source term $f$ is $p$-integrable for values of $p$ larger than $n+2$.  
In this setting, we obtain that such solutions enjoy $C^{1+\alpha^{\prime},\frac{1+\alpha^{\prime}}{2}}$ estimates, where $\alpha^{\prime}\in (0,1)$ is given by
\begin{equation}\label{4.1}
\alpha^{\prime}=\min\left\{\alpha^{-},1-\frac{n+2}{p}\right\}.
\end{equation}
where this constant should be understood in the following sense,
\begin{equation*}
\left\{
\begin{array}{lclc}
u\in C^{2-\frac{n+2}{p},1-\frac{n+2}{2p}},& & \text{if} & 1-\frac{n+2}{p}<\alpha, \\
u\in C^{1+\nu,\frac{1+\nu}{2}}\,\,\, \text{for any }\,\,\, 0<\nu<\alpha, &   & \text{if} & 1-\frac{n+2}{p}\geq\alpha,
\end{array}
\right.
\end{equation*}
To prove this result, revisiting Lemma \ref{Lemma3.1} we obtain the following adaptation.
\begin{lemma}\label{Lemma4.1}
Fix $n+2<p<\infty$ and $0<\tilde{\alpha}<\alpha$. Let $u$ be a normalized solution to \eqref{1.1} with $\beta,g\in C^{\alpha,\frac{\alpha}{2}}(\overline{{{Q}_{1}^{+}}})$. There exist a universal constant $\delta>0$ and a radius $0<\rho\leq\frac{1}{2}$ such that, if 
\[
\max\left\{ \left(\intav{Q_{1}^{+}}\Phi_{F}(x,t)^{p}dxdt\right)^{\frac{1}{p}},\,\|f\|_{L^{p}(Q^{+}_{1})}  \right\} \le \delta,
\]
then we can find an affine profile $\ell(x)=a\cdot x+b$ such that:
\begin{itemize}
\item[(i)] $\displaystyle\sup_{\overline{Q^{+}_{\rho}}}|u-\ell|\leq \rho^{1+\tilde{\alpha}}$;
\item[(ii)] $\beta(0,0)\cdot a=g(0,0)$;
\item[(iii)] $|a|+|b|\leq \mathfrak{C}$, for some universal constant $\mathfrak{C}>0$. 
\end{itemize}
\end{lemma}
\begin{proof}
The proof follows exactly as in Lemma \ref{Lemma3.1}, with the following modifications. We adjust the choice of $\rho$ and $\varepsilon$ to
\[
\rho=\min\left\{\frac{1}{2},\left(\frac{1}{2\mathrm{C}_{0}}\right)^{\frac{1}{\alpha-\tilde{\alpha}}}\right\}\,\,\, \text{and}\,\,\, \varepsilon=\frac{1}{2}\rho^{1+\tilde{\alpha}}.
\]
In this case, the desired conclusion follows.
\end{proof}

With this approximation by an affine function, we are now in a position to prove the main result of this section.
\begin{theorem}[\bf Optimal regularity of the gradient]\label{Theorem4.2}
Let $u$ be a viscosity solution to \eqref{1.1}, where $\beta,g\in C^{\alpha,\frac{\alpha}{2}}(\overline{Q^{*}_{1}})$ and $f\in L^{p}(Q^{+}_{1})\cap C^{0}(Q^{+}_{1})$ for $n+2<p<\infty$. There exists constants $r_{0}>0$ and $\zeta_{0}>0$ depending only on where $\mathrm{C}$ is a positive constant depending only on $n$, $\lambda$, $\Lambda$, $\mu_{0}$, $p$, $\alpha$, and $\|\beta\|_{C^{\alpha,\frac{\alpha}{2}}(\overline{Q_{1}^{*}})}$, such that if
\begin{equation}\label{4.2}
\sup_{0<r\leq r_0} ~\sup_{(x,t)\in Q^+_{1/2}}\left(\intav{Q_{r}(x_,t)\cap Q_{1}^{+}}\Phi_{F}((y,s);(x,t))^{p}dyds\right)^{\frac{1}{p}}\leq \zeta_{0},
\end{equation}
then $u\in C^{1+\alpha^{\prime},\frac{1+\alpha^{\prime}}{2}}(\overline{Q^{+}_{\frac{1}{2}}})$ with the following estimate
\[
\|u\|_{C^{1+\alpha^{\prime},\frac{1+\alpha^{\prime}}{2}}(\overline{Q^{+}_{\frac{1}{2}})}}\leq\mathrm{C}\Bigl(\|u\|_{L^{\infty}(Q^{+}_{1})}+\|f\|_{L^{p}(Q^{+}_{1})}+\|g\|_{C^{\alpha,\frac{\alpha}{2}}(\overline{Q^{*}_{1}})}\Bigr),
\]
\end{theorem}
where $\mathrm{C}$ is a positive constant depending only on $n$, $\lambda$, $\Lambda$, $\mu_{0}$, $p$, $\alpha$, and $\|\beta\|_{C^{\alpha,\frac{\alpha}{2}}(\overline{Q_{1}^{*}})}$.
\begin{proof}
By normalization, we may assume without loss of generality that
$$
\|u\|_{L^{\infty}(Q^{+}_{1})}\leq 1,\quad  \|f \|_{L^{p}(Q^{+}_{1})}\leq \delta, \quad g(0,0)=0, \quad \text{and} \quad  \|g\|_{C^{\alpha,\frac{\alpha}{2}}(\overline {Q^{*}_{1}})}\leq 1,
$$
where $\delta>0$ is the constant from Lemma \ref{Lemma4.1}, and we set $\zeta_{0}=\delta$ in \eqref{4.2}. We first establish regularity at points on the boundary of the flat portion of the boundary of the half-cylinder $Q^{+}_{\frac{1}{2}}$. By a translation argument, it suffices to establish the result at the origin $(0,0)$. In this case, we claim that there exists a sequence of affine functions 
$(\ell_{k})_{k\geq 0}$ of the form
\[
\ell_{k}(x) = a_{k}\cdot x + b_{k},
\]
satisfying
\begin{equation}\label{4.3}
\sup_{\overline{Q_{\rho^{k}}^{+}}}|u-\ell_{k}|\leq \rho^{k(1+\alpha^{\prime})}.
\end{equation}
Moreover, the coefficients of the affine functions $\ell_{k}$ satisfy for all $k\geq 0$
\begin{itemize}
\item[(i)]  $\beta(0,0)\cdot a_{k}=0$;
\item[(ii)] $\rho^{k}|a_{k+1}-a_{k}|+|b_{k+1}-b_{k}|\leq \mathfrak{C}\rho^{k(1+\alpha^{\prime})}$,
\end{itemize}
where $\rho$ and $\mathfrak{C}$ are the constants for Lemma \ref{Lemma4.1} when we chosen $\tilde{\alpha}=\alpha^{\prime}$. The proof of the claim follows by induction on $k$. Indeed, the base case $k=0$ holds by setting $\ell_{0}=0$ and using the fact that $u$ is a normalized solution. Now, assuming that the result holds for some $k\geq 0$, we have that the function
\[
u_{k}(x,t)=\frac{u(\rho^{k}x,\rho^{2k}t)-\ell_{k}(\rho^{k}x)}{\rho^{k(1+\alpha^{\prime})}}
\]
is a normalized solution (from \eqref{4.3}) to the following problem
\begin{equation*}
\left\{
\begin{array}{rclcl}
F_{k}(D^2u_{k},x,t) - (u_{k})_{t} &=& f_{k}(x,t) & \text{in} & Q^{+}_{1}, \\
\beta_{k} \cdot Du_{k}&=& g_{k}(x,t) & \text{on} & Q_{1}^{*},
\end{array}
\right.
\end{equation*}
where
$$
\left\{
\begin{array}{rcl}
F_{k}(\mathrm{M},x,t) & \defeq & \rho^{k(1-\alpha^{\prime})}F\left(\frac{1}{\rho^{k(1-\alpha^{\prime})}} \mathrm{X},\rho^{k}x,\rho^{2k}t \right) \\
f_{k}(x,t)&\defeq&\rho^{k(1-\alpha^{\prime})}f(\rho^{k}x,\rho^{2k}t)\\
\beta_{k}(x,t) & \defeq & \beta(\rho^{k}x,\rho^{2k}t)\\
g_{k}(x,t)& \defeq &\rho^{-k\alpha^{\prime}}( g(\rho^{k}x,\rho^{2k}t)-\beta(\rho^{k}x,\rho^{2k}t)\cdot a_{k}).
\end{array}
\right.
$$
We claim that $u_{k}$ satisfies the hypotheses of Lemma \ref{Lemma4.1}. Indeed, clearly $F_{k}\in \mathcal{E}(\lambda,\Lambda)$, and by condition \eqref{4.2} we have that
\[
\left(\intav{Q_{1}^{+}}\Phi_{F_{k}}(x,t)^{p}dxdt\right)^{\frac{1}{p}}=\left(\intav{Q_{\rho^{k}}^{+}}\Phi_{F}(y,s)^{p}dyds\right)^{\frac{1}{p}}\leq \zeta_{0}=\delta.
\]
Moreover, by the definition of $\alpha^{\prime}$ we have that
\[
\|f_{k}\|_{L^{p}(Q_{1}^{+})}=\rho^{k\left((1-\alpha^{\prime})-\frac{n+2}{p}\right)}\|f\|_{L^{p}(Q_{\rho^{k}}^{+})}\leq \|f\|_{L^{p}(Q_{1}^{+})}\leq \delta.
\]
Now, regarding the boundary data, it is not difficult to verify that $\beta_{k}\in C^{\alpha,\frac{\alpha}{2}}(\overline{Q_{1}^{*}})$ and 
\[
\|\beta_{k}\|_{C^{\alpha,\frac{\alpha}{2}}(\overline{Q_{1}^{*}})}=\|\beta\|_{L^{\infty}(\overline{Q_{\rho^{k}}^{*}})}+\rho^{k\alpha}[\beta]_{\alpha,\overline{Q_{\rho^{k}}^{*}}}\leq \|\beta\|_{C^{\alpha,\frac{\alpha}{2}}(\overline{Q_{1}^{*}})}.
\]
Concerning the boundary source term $g_{k}$, note that $g_{k}(0,0)=0$, due to the induction hypothesis and the fact that $g(0,0)=0$. Therefore, in this case we have
\begin{equation}\label{4.4}
\|g_{k}\|_{L^{\infty}(\overline{Q_{1}^{*}})}\leq [g_{k}]_{\alpha,\overline{Q_{1}^{*}}}.
\end{equation}
Thus, in order to obtain a universal control of the Hölder norm of $g_{k}$, it suffices to control its seminorm. To this end, one can verify that
\begin{equation}\label{4.5}
[g_{k}]_{\alpha,\overline{Q_{1}^{*}}}\leq \rho^{k(\alpha-\alpha^{\prime})}\left([g]_{\alpha,\overline{Q_{\rho^{k}}^{*}}}+[\beta]_{\alpha,\overline{Q_{\rho^{k}}^{*}}}|a_{k}|\right)\leq  [g]_{\alpha,\overline{Q_{1}^{*}}}+ \frac{\mathfrak{C}}{1-\rho^{\alpha^{\prime}}}[\beta]_{\alpha,\overline{Q_{1}^{*}}},
\end{equation}
since $\rho<1$ and by the induction hypothesis,
\begin{equation*}
|a_{k}|\leq \sum_{j=0}^{k-1}|a_{j+1}-a_{j}|\leq \mathfrak{C}\sum_{j=0}^{k-1}\rho^{\alpha^{\prime}j}\leq \frac{\mathfrak{C}}{1-\rho^{\alpha^{\prime}}}.
\end{equation*}
Therefore, from \eqref{4.4} and \eqref{4.5} we conclude that $g_{k}\in C^{\alpha,\frac{\alpha}{2}}(\overline{Q_{1}^{*}})$ with the following estimate,
\[
\|g_{k}\|_{C^{\alpha,\frac{\alpha}{2}}(\overline{Q_{1}^{*}})}\leq\mathrm{C}(\alpha^{\prime},\mathfrak{C},\rho) \|g\|_{C^{\alpha,\frac{\alpha}{2}}(\overline{Q_{1}^{*}})},
\]
and in this case the bound is universal, due to its dependence only on universal constants.Thus, we are under the assumptions of Lemma \ref{Lemma4.1}, and by applying this result we obtain that there exists an affine function $\bar{\ell}(x)=\bar{a}\cdot x+\bar{b}$ such that
\begin{equation}\label{4.6}
\sup_{\overline{Q^{+}_{\rho}}}|u_{k}-\bar{\ell}|\leq \rho^{1+\alpha^{\prime}},
\end{equation}
and the coefficients $\bar{a}$ and $\bar{b}$ satisfy,
\begin{equation}\label{4.7}
\beta_{k}(0,0)\cdot \bar{a}=0 \,\,\, \text{(since $g_{k}(0,0)=0$)}\,\,\, \text{and}\,\,\, |\bar{a}|+|\bar{b}|\leq \mathfrak{C}.
\end{equation}
Thus, defining $a_{k+1}=a_{k}+\rho^{k\alpha^{\prime}}\bar{a}$ and $b_{k+1}=b_{k}+\rho^{k(1+\alpha^{\prime})}\bar{b}$, the affine profile $\ell_{k+1}$ is such that, upon rescaling estimate \eqref{4.6}, we obtain that
\[
\sup_{\overline{Q_{\rho^{k+1}}^{+}}}|u-\ell_{k+1}|\leq \rho^{(k+1)(1+\alpha^{\prime})}.
\]
Moreover, the boundedness of the coefficients of $\bar{\ell}$ in \eqref{4.7} implies that
\[
\rho^{k}|a_{k+1}-a_{k}|+|b_{k+1}-b_{k}|=\rho^{k(1+\alpha^{
\prime})}(|\bar{a}|+|\bar{b}|)\leq \mathfrak{C}\rho^{k(1+\alpha^{\prime})}
\]
and by the definition of $\beta_{k}$ and the induction hypothesis implies that
\[
\beta(0,0)\cdot a_{k+1}=\beta_{k}(0,0)\cdot a_{k}+\rho^{k\alpha^{\prime}}\beta_{k}(0,0)\cdot \mathrm{a}=0.
\]
This proves the case $k+1$, thereby establishing the claim by induction.

Now, concerning this sequence of affine functions, we observe that it is convergent, since the sequences of coefficients $(a_{k})_{k\in\mathbb{N}}$ and $(b_{k})_{k\in\mathbb{N}}$ are Cauchy sequences and, by construction, they converge to vectors $a_{\infty}$ and $b_{\infty}$ with the following rate of convergence
\begin{equation}\label{4.8}
|a_{k}-a_{\infty}|\leq \frac{\mathfrak{C}}{1-\rho^{\alpha^{\prime}}}\rho^{k\alpha^{\prime}}\,\,\, \text{and}\,\,\, |b_{k}-b_{\infty}|\leq \frac{\mathfrak{C}}{1-\rho^{1+\alpha^{\prime}}}\rho^{k(1+\alpha^{\prime})}.
\end{equation}
Therefore, defining the affine function $\ell_{\infty}(x)=a_{\infty}\cdot x+b_{\infty}$, given $r\in (0,\rho)$ there exists $k\in\mathbb{N}$ such that $\rho^{k+1}<r\leq \rho^{k}$ and, in this case, using \eqref{4.3} and \eqref{4.8} it follows that
\begin{eqnarray*}
\sup_{\overline{Q_{r}^{+}}}|u-\ell_{\infty}|&\leq& \sup_{\overline{Q_{\rho^{k}}^{+}}}|u-\ell_{k}| +\sup_{\overline{Q_{\rho^{k}}^{+}}}|\ell_{k}-\ell_{\infty}|\\
&\leq&\rho^{k(1+\alpha^{\prime})}+\left(\frac{1}{1-\rho^{\alpha^{\prime}}}+\frac{1}{1-\rho^{1+\alpha^{\prime}}}\right)\mathfrak{C}\rho^{k(1+\alpha^{\prime})}\\
&\leq& \mathrm{C}_{0}r^{1+\alpha^{\prime}},
\end{eqnarray*}
where $\mathrm{C}_{0}=\left(\frac{1}{1-\rho^{\alpha^{\prime}}}+\frac{1}{1-\rho^{1+\alpha^{\prime}}}\right)\frac{\mathfrak{C}}{\rho^{1+\alpha^{\prime}}}$. Therefore, $u$ is of class $C^{1+\alpha^{\prime},\frac{1+\alpha^{\prime}}{2}}$ at $(0,0)$ with $Du(0,0)=a_{\infty}$ and $u(0,0)=b_{\infty}$, which proves the initial claim. Finally, combining this with optimal interior estimates (cf. \cite[Section 5]{DaST17}) it follows that $u\in C^{1+\alpha^{\prime},\frac{1+\alpha^{\prime}}{2}}(\overline{Q_{1/2}^{+}})$, thus completing the proof.
\end{proof}
As a consequence of the optimal regularity of the gradient, we obtain the following corollary, which classifies the regularity of the model \eqref{1.1} under different perspectives on the boundary data and the source term.
\begin{corollary}
Let $u$ be a viscosity solution to \eqref{1.1}. The following assertions hold:
\begin{itemize}
\item[(a)] Assume that $\beta,g\in C^{\alpha,\frac{\alpha}{2}}(\overline{Q_{1}^{*}})$ for some $0<\alpha<1$ and $f\in L^{\frac{n+2}{1-\alpha}}(Q_{1}^{+})\cap C^{0}(Q_{1}^{+})$. There exists a universal constants $r_{0}$ and $\zeta_{0}>0$ such that if \eqref{4.2} holds,then $u\in C^{1+\alpha^{-},\frac{1+\alpha^{-}}{2}}(\overline{Q_{\frac{1}{2}}^{+}})$. Precisely, $u\in C^{1+\nu,\frac{1+\nu}{2}}(\overline{Q_{\frac{1}{2}}^{+}})$ for any $0<\nu<\alpha$.
\item [(b)] Assume that $\beta,g\in C^{1,\frac{1}{2}}(\overline{Q_{1}^{*}})$ and $f\in L^{p}(Q_{1}^{+})\cap C^{0}(Q_{1}^{+})$ for some $n+2<p<\infty$. There exists a universal constants $r_{0}$ and $\zeta_{0}>0$ such that if \eqref{4.2} holds, then $u\in C^{2-\frac{n+2}{p},1-\frac{n+2}{2p}}(\overline{Q_{\frac{1}{2}}^{+}})$. 
\end{itemize}
\end{corollary}
\begin{proof}
For item (a), it suffices to observe that $p=\frac{n+2}{1-\alpha}>n+2$ since $0<\alpha<1$, and thus the result follows from Theorem \ref{Theorem4.2}, because in this configuration $1-\frac{n+2}{p}=\alpha$.

On the other hand, for item (b), note that in the case $\alpha=1$, for every $n+2<p<\infty$ one has $1-\frac{n+2}{p}<1$, and the conclusion follows by applying Theorem \ref{Theorem4.2}.
\end{proof}

\section{Higher Order Regularity: Schauder Estimates}\label{Section5}

In this part, we are interested in establishing higher regularity for viscosity solutions of \eqref{1.1} within the framework of H\"{o}lder spaces, namely Schauder-type estimates. Following the polynomial approach developed in the previous sections, we aim to establish estimates when the source term and the oscillation of the coefficients of the governing operator enjoy a H\"{o}lder modulus of continuity, at least in a weaker (integral) sense. More precisely, throughout this section we assume the following structural conditions:
\begin{itemize}
\item[{\bf (H1)}] ({\bf Regularity of source term})  
We assume that the source term $f$ belongs to $C^{0,\alpha}(\overline{Q^{+}_{1}})$ in the $L^{n+2}$ sense; that is, there exists a constant $\mathrm{C}_{f}>0$ such that 
\[
\left(\intav{Q^{+}_{1}\cap Q_{r}(x_{0},t_{0})} |f(x,t)-f(x_{0},t_{0})|^{n+2}\,dxdt \right)^{\frac{1}{n+2}} \leq \mathrm{C}_{f}r^{\alpha},
\quad \forall (x_{0},t_{0})\in Q^{+}_{1}, \, r\in(0,1).
\]
\item[{\bf (H2)}] ({\bf Oscillation of the coefficients})  
We assume that the oscillation of the coefficients $\tilde{\Phi}_{F}(\cdot, x_{0}) \in C^{0,\alpha}(Q^{+}_{1})$ in the $L^{n+2}$ sense; i.e., there exists $\mathrm{C}_{F}>0$ such that 
\[
\left(\intav{Q^{+}_{1}\cap Q_{r}(x_{0},t_{0})} 
\tilde{\Phi}_{F}((x,t);(x_{0},t_{0}))^{n+2}\,dxdt\right)^{\frac{1}{n+2}} 
\leq \mathrm{C}_{F}\,r^{\alpha},
\quad \forall (x_{0},t_{0})\in Q^{+}_{1}, \, r\in(0,1).
\]
\item[{\bf (H3)}] ({\bf $C^{2+\alpha,1+\frac{\alpha}{2}}$-type estimates for translated problems})  
For each $(x_{0},t_{0})\in \overline{Q^{+}_{1/2}}$, consider the auxiliary problem
\[
\left\{
\begin{array}{rclcl}
F(D^{2}\mathfrak{h}+\mathrm{M},x_{0},t_{0})-\mathfrak{h}_{t} & = & 0 & \text{in} & Q^{+}_{\frac{7}{8}}, \\[2pt]
\beta\!\cdot\! D\mathfrak{h} & = & g(x,t) & \text{on} & Q^{*}_{\frac{7}{8}}.
\end{array}
\right.
\]
We assume that this problem admits a solution 
$\mathfrak{h}\in C^{2+\alpha,\frac{2+\alpha}{2}}(\overline{Q^{+}_{\frac{2}{3}}}) 
\cap C^{0}(Q^{+}_{\frac{7}{8}}\cup Q^{*}_{\frac{7}{8}})$
whenever $\beta,,g\in C^{1+\alpha,\frac{1+\alpha}{2}}(\overline{Q^{*}_{\frac{7}{8}}})$, 
and that the following estimate holds:
\[
\|\mathfrak{h}\|_{C^{2+\alpha,\frac{2+\alpha}{2}}(\overline{Q^{+}_{\frac{2}{3}}})}
\leq \mathrm{C}_{\star}\!
\left(\|\mathfrak{h}\|_{L^{\infty}(Q^{+}_{\frac{7}{8}})}
+\|g\|_{C^{1+\alpha,\frac{1+\alpha}{2}}(\overline{Q^{*}_{\frac{7}{8}}})}\right),
\]
for a universal constant $\mathrm{C}_{\star}>0$.
\end{itemize}

\begin{remark}\label{remark5.1}
On the previous conditions:
\begin{itemize}
\item[\checkmark] Concerning the oscillation function $\tilde{\Phi}_{F}$ appearing in assumption \textbf{(H2)}, we note that it controls the oscillation $\Phi_{F}$ introduced in Section~\ref{Section2}. More precisely, one has $
\Phi_{F}\leq \tilde{\Phi}_{F}$.
\item [\checkmark] As examples of operators satisfying condition \textbf{(H3)}, we may mention concave or convex operators (cf. \cite[Theorem 5.8]{CM21}). Moreover, by suitably adapting the arguments therein, it is possible to show that the class of quasiconvex and quasiconcave operators also satisfies \textbf{(H3)} (cf. \cite[Theorem 7.1]{Gof24} and \cite[Theorem 1.2]{BRS}).
\end{itemize} 
\end{remark}

In this setting, we can state the key lemma ensuring quadratic approximations.
\begin{lemma}\label{Lemma5.2}
Let $u$ be a normalized solution to \eqref{1.1} with $\beta,g\in C^{1+\alpha,\frac{1+\alpha}{2}}(\overline{{{Q}_{1}^{*}}})$ and $\bar{\alpha}\in (0,\alpha)$. There exist universal constants $\delta>0$ and $0<\rho\leq\frac{1}{2}$ such that, if 
\[
\max\left\{ \left(\intav{Q_{1}^{+}}\tilde{\Phi}_{F}(x,t)^{n+2}dxdt\right)^{\frac{1}{n+2}},\,\|f\|_{L^{n+2}(Q^{+}_{1})}  \right\} \le \delta,
\]
then there exists a quadratic polynomial $P(x,t)=\frac{1}{2}x^{T}\mathrm{M}x+a\cdot x+bt+c$ such that
\begin{equation*}
\displaystyle\sup_{\overline{Q^{+}_{\rho}}}|u-\mathrm{P}|\leq \rho^{2+\bar{\alpha}}.
\end{equation*}
Moreover, the coefficients of the polynomial $\mathrm{P}$ satisfies
\begin{itemize}
\item[(i)] $F(\mathrm{M},0,0)=b$;
\item[(ii)] $\beta(0,0)\cdot a=g(0,0)$;
\item[(iii)] For $1\leq i\leq n-1$  
\begin{equation*}
\sum_{j=1}^{n}\!\left(D_{i}\beta_{j}(0,0)a_{j}+\beta_{j}(0,0)\mathrm{M}_{ij}\right)= D_{i}g(0,0);
\end{equation*}
\item[iv.] $|a|+|b|+|c|+\|\mathrm{M}\|\leq \mathfrak{C}$, for some universal constant $\mathfrak{C}>0$. 
\end{itemize}
\end{lemma}

\begin{proof}
The proof follows the same lines as that of Lemma~\ref{Lemma3.1}; therefore, we detail only the necessary modifications. Fix $\varepsilon>0$, to be chosen later. By applying Lemma~\ref{Approx}, we obtain the existence of a function $v$ solving
\begin{equation}\label{limitproblem}
\left\{
\begin{array}{rclcl}
F(D^2 \mathrm{v},0,0) - \mathrm{v}_{t} &=& 0 & \text{in} & Q^{+}_{\frac{3}{4}}, \\
\beta \cdot D\mathrm{v}&=& g(x,t) & \text{on} & Q_{\frac{3}{4}}^{*}, \\
\mathrm{v} &=& u & \text{on} & \partial_{p}Q^{+}_{\frac{3}{4}}\setminus Q_{\frac{3}{4}}^{*}
\end{array}
\right.
\end{equation}
and satisfy 
\begin{equation}\label{5.1}
\|u - \mathrm{v}\|_{L^{\infty}\left(Q^{+}_{\frac{3}{4}}\right)} \le \varepsilon.
\end{equation}
Invoking hypothesis {\bf (H3)}, we have $v \in C^{2+\alpha,\,\frac{2+\alpha}{2}}\big(\overline{Q^{+}_{\frac{2}{3}}}\big)$, and the following estimate holds
\begin{equation*}
\|\mathrm{v}\|_{C^{2+\alpha,\frac{2+\alpha}{2}}(\overline{Q^{+}_{\frac{2}{3}}})}\leq \mathrm{C}_{\star} \!
\left(\|\mathrm{v}\|_{L^{\infty}(Q^{+}_{\frac{7}{8}})}
+\|g\|_{C^{1+\alpha,\frac{1+\alpha}{2}}(\overline{Q^{*}_{\frac{7}{8}}})}\right).
\end{equation*}
Consequently, by the A.B.P.T estimate \ref{ABPTestimates},
\begin{equation}\label{6.2}
\|\mathrm{v}\|_{C^{2+\alpha,\frac{2+\alpha}{2}}(\overline{Q^{+}_{\frac{2}{3}}})}\leq \mathrm{C}^{\prime}\left(\mathrm{C}_{\star},n,\lambda,\Lambda, \mu_{0}, \|\beta\|_{C^{1+\alpha,\frac{1+\alpha}{2}}(\overline{Q^{+}_{\frac{2}{3}}})},\|g\|_{C^{1+\alpha,\frac{1+\alpha}{2}}(\overline{Q^{+}_{\frac{2}{3}}})} \right).
\end{equation}
Now, setting $\mathrm{M}=D^{2}\mathrm{v}(0,0)$, $a = D\mathrm{v}(0,0)$, $b=\mathrm{v}_{t}(0,0)$, and $c = \mathrm{v}(0,0)$, the quadratic polynomial $P$ is well defined, and for every $(x,t) \in Q^{+}_{\frac{3}{4}}$ we have that
\begin{eqnarray}
|\mathrm{v}(x,t)-P(x,t)|&\leq&\left|\mathrm{v}(x,t)-\mathrm{v}(0,t)-D\mathrm{v}(0,t)\cdot x-\frac{1}{2}x^{T}D^{2}\mathrm{v}(0,t)x\right|\nonumber\\
&+&\frac{1}{2}|x^{T}(D^{2}\mathrm{v}(0,t)-D^{2}\mathrm{v}(0,0))x|+|\mathrm{v}(0,t)-\mathrm{v}(0,0)-\mathrm{v}_{t}(0,0)t|\nonumber\\
&+&|D\mathrm{v}(0,t)-D\mathrm{v}(0,0)||x|\nonumber\\
&\leq&[D^{2}\mathrm{v}]_{\alpha,\overline{Q^{+}_{\frac{2}{3}}}}|x|^{2}(|x|^{\alpha}+|t|^{\frac{\alpha}{2}})+[\mathrm{v}_{t}]_{\alpha,\overline{Q^{+}_{\frac{2}{3}}}}|t|^{1+\frac{\alpha}{2}}+[D\mathrm{v}]_{1+\alpha,\overline{Q^{+}_{\frac{2}{3}}}}|x||t|^{\frac{1+\alpha}{2}}\label{6.3}\\
&\leq&\mathrm{C}(|x|^{2+\alpha}+|t|^{1+\frac{+\alpha}{2}}+|x||t|^{\frac{1+\alpha}{2}}+|x|^{2}|t|^{\frac{\alpha}{2}})\nonumber\\
&\leq&\mathrm{C}_{1}d_{\text{par}}((x,t),(0,0))^{2+\alpha},\label{6.4}
\end{eqnarray}
where in \eqref{6.3} we use the estimate \eqref{6.2}. Thus, the estimate \eqref{6.4} implies
\begin{equation}\label{6.5}
\sup_{\overline{Q^{+}_{r}}
}|\mathrm{v}-P|\leq \mathrm{C}_{1}r^{2+\alpha}, \,\, \forall r\in \left(0,\frac{2}{3}\right).
\end{equation}
We are now in a position to choose $\rho$ and $\varepsilon$ such that by setting
\[
\rho=\min\left\{\frac{1}{2},\left(\frac{1}{2\mathrm{C}_{1}}\right)^{\frac{1}{\alpha-\bar{\alpha}}}\right\}\,\,\, \text{and}\,\,\, \varepsilon=\frac{1}{2}\rho^{2+\bar\alpha}.
\]
Consequently, the constant $\delta > 0$ is well determined by the Approximation Lemma~\ref{Approx}. As a consequence of these choices, together with \eqref{5.1} and \eqref{6.5}, it follows that
\begin{equation*}
\sup_{\overline{Q_{\rho}^{+}}}|u-P|\leq \rho^{2+\bar{\alpha}}.
\end{equation*}
Moreover, by the choice of the coefficients of the quadratic polynomial $P$, it follows directly from \eqref{limitproblem} that conditions (i) and (ii) are satisfied, and by estimate \eqref{6.2} we obtain the universal boundedness condition for the coefficients, namely (iv). Finally, using the boundary equation on the flat boundary $\mathrm{T}_{\frac{7}{8}}$ and the fact that $v \in C^{2+\alpha,\,\frac{2+\alpha}{2}}\big(\overline{Q^{+}_{\frac{2}{3}}}\big)$, it follows that 
\[
\beta(x,t)\cdot D\mathrm{v}=g(x,t),\,\, \forall (x,t)\in Q^{*}_{\frac{2}{3}}.
\]
By differentiating both sides in the $i$-th direction ($i = 1, \ldots, n-1$) and evaluating at $(x,t) = (0,0)$, the compatibility condition~(iv) follows.
\end{proof}

We are now in a position to prove the main theorem of this section.

\begin{theorem}[{\bf Schauder estimates}]\label{Theorem5.3}
Let $u$ be a solution to the problem \eqref{1.1}. Assume that the structural conditions {\bf(H1)-(H3)} hold. Then, there exist universal constants $r_{0}>0$ and $\zeta_{0}>0$ such that if
\[
\sup_{0<r\leq r_0} ~\sup_{(x,t)\in Q^+_{1/2}}\left(\intav{Q_{r}(x_,t)\cap Q_{1}^{+}}\tilde{\Phi}_{F}((y,s);(x,t))^{n+2}dyds\right)^{\frac{1}{n+2}}\leq \zeta_{0}
\]
then $u\in C^{2+\alpha^{-},\frac{2+\alpha^{-}}{2}}(\overline{Q^{+}_{\frac{1}{2}}})$. Precisely, for any $\gamma\in (0,\alpha)$, we have $u\in C^{2+\gamma,\frac{2+\gamma}{2}}(\overline{Q^{+}_{\frac{1}{2}}})$ and the following estimate holds
\[
\|u\|_{C^{2+\gamma,\frac{2+\gamma}{2}}(\overline{Q^{+}_{\frac{1}{2}})}}\leq\mathrm{C}\Bigl(\|u\|_{L^{\infty}(Q^{+}_{1})}+\|f\|_{C^{0,\alpha}(\overline{Q^{+}_{1}})}+\|g\|_{C^{1+\alpha,\frac{1+\alpha}{2}}(\overline{Q^{*}_{1}})}\Bigr). 
\]
\end{theorem}
\begin{proof}
For convenience, by normalization and scaling arguments, we assume that $\|u\|_{L^{\infty}(Q_{1}^{+})} \leq 1$, 
$g(0,0)=f(0,0)=0$ and $Dg(0,0)=0$ (cf. \cite[Theorem 5.8]{CM21}). Fix $\gamma\in(0,\alpha)$ and the constants $\delta>0$ and $\rho$ from the Lemma \ref{Lemma5.2}. We first establish the desired regularity at points on the flat boundary 
$\overline{Q^{*}_{\frac{1}{2}}}$. By a standard localization argument, it is enough to establish the result at the origin $(0,0)$. More precisely, we claim that there exists a sequence of quadratic polynomials 
$(P_{k})_{k\geq -1}$ of the form
\[
P_{k}(x,t)
=
\frac{1}{2}x^{T}\mathrm{M}_{k}x
+ a_{k}\cdot x
+ b_{k}t
+ c_{k},
\]
satisfying
\begin{equation}\label{6.7}
\sup_{\overline{Q_{\rho}^{k}}}|u-P_{k}|\leq \rho^{2+\gamma}.
\end{equation}
Moreover, the coefficients of the quadratic polynomials $P_{k}$ satisfy for all $k\geq 0$
\begin{itemize}
\item[(a)] $F(\mathrm{M}_{k},0,0)=b_{k}$;
\item[(b)]  $\beta(0,0)\cdot a_{k}=0$;
\item[(c)] For any $1\leq i\leq n-1$
\begin{equation*}
\sum_{j=1}^{n}\!\left(D_{i}\beta_{j}(0,0)a_{j}^{(k)}+\beta_{j}(0,0)\mathrm{M}_{ij}^{(k)}\right)= 0,
\end{equation*}
where $a^{(k)}_{j}$ denotes the $j$-th component of the vector $a_{k}$ and $\mathrm{M}_{ij}^{(k)}$ denotes the $(i,j)$-th entry of the matrix $\mathrm{M}_{k}$;
\item[(d)] $|c_{k}-c_{k-1}|+\rho^{(k-1)}|a_{k}-a_{k-1}|+\rho^{2(k-1)}(|b_{k}-b_{k-1}|+\|\mathrm{M}_{k}-\mathrm{M}_{k-1}\|)\leq \mathfrak{C}\rho^{(2+\gamma)(k-1)}$, for some universal constant $\mathfrak{C}>0$. 
\end{itemize}
By a scaling argument (cf. \cite[Theorem 8.1]{CC}), we may assume that
\begin{equation}\label{6.8}
\mathrm{C}_{f}\leq \frac{\delta}{2^{\frac{n+3}{n+2}}|Q_{1}^{+}|^{\frac{1}{n+2}}}
\quad \text{and} \quad
\mathrm{C}_{F}\leq \frac{(1-\rho^{\gamma})\delta}{4(1+2^{\frac{n+1}{n+2}}|Q_{1}^{+}|^{\frac{1}{n+2}})(1+\mathfrak{C}-\rho^{\gamma})}.    
\end{equation}
The proof of the result is carried out by induction on $k$. Setting $P_{0}=P_{-1}=0$, the case $k=0$ follows immediately.
Assume by induction that the statement holds up to some $k\geq 0$. Then the rescaled profile
$$
u_{k}(x,t)=\frac{(u-P_{k})(\rho^{k}x,\rho^{2k}t)}{\rho^{k(2+\gamma)}}
$$
satisfies $\|u_{k}\|_{L^{\infty}(Q_{1}^{+})}\leq 1$ (by \eqref{6.7}) and solves
\begin{equation*}
\left\{
\begin{array}{rclcl}
F_{k}(D^2u_{k},x,t) - (u_{k})_{t} &=& f_{k}(x,t) & \text{in} & Q^{+}_{1}, \\
\beta_{k} \cdot Du_{k}&=& g_{k}(x,t) & \text{on} & Q_{1}^{*},
\end{array}
\right.
\end{equation*}
where
$$
\left\{
\begin{array}{rcl}
F_{k}(\mathrm{X},x,t) & \defeq & \dfrac{1}{\rho^{k\gamma}}[F\left(\rho^{k\gamma} \mathrm{X}+\mathrm{M}_{k},\rho^{k}x,\rho^{2k}t \right)-F(\mathrm{M}_{k},\rho^{k}x,\rho^{2k}t)] \\
f_{k}(x,t)&\defeq&\dfrac{1}{\rho^{k\gamma}}[f(\rho^{k}x,\rho^{2k}t)-F(\mathrm{M}_{k},\rho^{k}x,\rho^{2k}t)+b_{k}]\\
\beta_{k}(x,t) & \defeq & \beta(\rho^{k}x,\rho^{2k}t)\\
g_{k}(x,t)& \defeq &\dfrac{1}{\rho^{-k(1+\gamma)}}[ g(\rho^{k}x,\rho^{2k}t)-\beta(\rho^{k}x,\rho^{2k}t)\cdot DP_{k}(\rho^{k}x,\rho^{2k}t)].
\end{array}
\right.
$$
We will show that this problem satisfies the assumptions of Lemma \ref{Lemma6.1}. In effect, first note that 
\begin{eqnarray*}
\Phi_{F_{k}}(x,t)
&=& \sup_{\mathrm{X}\in \mathrm{Sym}(n)} \frac{|F_{k}(\mathrm{X},x,t) - F_{k}(\mathrm{X},0,0)|}{1+\|\mathrm{X}\|} \nonumber\\
&\leq& \sup_{\mathrm{X}\in \mathrm{Sym}(n)} \Bigg| \frac{F(\rho^{k\gamma}\mathrm{X}+\mathrm{M}_{k},\rho^{k}x,\rho^{2k}t) - F(\mathrm{M}_{k},\rho^{k}x,\rho^{2k}t)}{\rho^{k\gamma}(1+\|\mathrm{X}\|)}  \nonumber\\
&-&\frac{F(\rho^{k\gamma}\mathrm{X}+\mathrm{M}_{k},0,0) - F(\mathrm{M}_{k},0,0)}{\rho^{k\gamma})(1+\|\mathrm{X}\|)} \Bigg|\nonumber\\
&\leq& \frac{\Phi_{F}(\rho^{k}x,\rho^{2k}t)}{\rho^{k\gamma}} \sup_{\mathrm{X}\in \mathrm{Sym}(n)} \!\left( \frac{\rho^{k\gamma}\|\mathrm{X}\|}{1+\|\mathrm{X}\|} + 2\frac{1+\|\mathrm{M}_{k}\|}{1+\|\mathrm{X}\|} \right) \nonumber\\
&\leq& 2\frac{\Phi_{F}(\rho^{k}x,\rho^{2k}t)}{\rho^{k\gamma}} (1+\|\mathrm{M}_{k}\|). 
\end{eqnarray*}
In this case, we obtain that
\begin{eqnarray}\label{6.9}
\left(\intav{Q_{1}^{+}}\Phi_{F_{k}}(x,t)^{n+2}\,dxdt\right)^{\frac{1}{n+2}}\leq 2(1+\|\mathrm{M}_{k}\|)\mathrm{C}_{F}\rho^{k(\alpha-\gamma)}\leq \frac{\delta}{2}<\delta,
\end{eqnarray}
This estimate follows from \eqref{6.8} and assumption \textrm{(H3)}, together with item (d) of the induction hypothesis, which yields
\[
\|\mathrm{M}_{k}\|\leq \sum_{j=1}^{k}\|\mathrm{M}_{j}-\mathrm{M}_{j-1}\|
\leq \mathfrak{C}\sum_{j=1}^{\infty}\rho^{(j-1)\gamma}
=\frac{\mathfrak{C}}{1-\rho^{\gamma}}.
\]
Now, regarding the source term $f_{k}$, by item (a) of the induction hypothesis, we have that
\begin{eqnarray*}
\|f_{k}\|_{L^{n+2}(Q_{1}^{+})}&\leq&\frac{2^{\frac{n+1}
{n+2}}|Q_{1}^{+}|^{\frac{1}{n+2}}}{\rho^{k\gamma}}\Bigg[
\left(\intav{Q_{\rho^{k}}^{+}}|f(x,t)-f(0,0)|^{n+2}\,dxdt\right)^{\frac{1}{n+2}}\\
&+&\left(\intav{Q_{1}^{+}}\Phi_{F_{k}}(x,t)^{n+2}\,dxdt\right)^{\frac{1}{n+2}}\Bigg]\\
&\stackrel{\eqref{6.9}}{\leq}&2^{\frac{n+1}
{n+2}}|Q_{1}^{+}|^{\frac{1}{n+2}}\rho^{k(\alpha-\gamma)}\left[\mathrm{C}_{f}+2(1+\|\mathrm{M}_{k}\|)\mathrm{C}_{F}\right]\\
&\leq&\delta,
\end{eqnarray*}
since assumptions \textrm{(H1)--(H2)} are satisfied and \eqref{6.8} holds. 

Now, it remains to analyze the regularity of the boundary data $\beta_{k}$ and $g_{k}$, as well as the dependence of their norms. Concerning the vector field $\beta_{k}$, it is clearly of class $C^{1+\alpha,\frac{1+\alpha}{2}}$
in $\overline{Q_{1}^{*}}$, since $\beta$ is and $0<\rho<1$, and moreover we have the following estimate:
\begin{equation*}
\|\beta_{k}\|_{C^{1+\alpha,\frac{1+\alpha}{2}}(\overline{Q_{1}^{*}})}\leq \|\beta\|_{C^{1+\alpha,\frac{1+\alpha}{2}}(\overline{Q_{1}^{*}})}.
\end{equation*}
On the other hand, since $g(0,0)=0$ and $Dg(0,0)=0$, it follows from this fact and items (b) and (c) of the induction hypothesis that the same holds for $g_{k}(0,0)$ and $Dg_{k}(0,0)$, respectively. In this case, it is sufficient to ensure $k$-independent bounds for
$[Dg_{k}]_{\alpha,\overline{Q_{1}^{*}}}$, and $[g_{k}]_{1+\alpha,\overline{Q_{1}^{*}}}$,  since once these bounds are established we have that
\begin{eqnarray*}
\|Dg_{k}\|_{L^{\infty}(\overline{Q_{1}^{*}})}\leq 2[Dg_{k}]_{\alpha,\overline{Q_{1}^{*}}}
\end{eqnarray*}
and
\begin{eqnarray*}
\|g_{k}\|_{L^{\infty}(Q_{1}^{*})}&\leq&\rho^{k(1+\alpha)}\left(2[Dg_{k}]_{\alpha,\overline{Q_{1}^{*}}}+[g_{k}]_{1+\alpha,\overline{Q_{1}^{*}}}\right)\leq \mathrm{C}\left([Dg_{k}]_{\alpha,\overline{Q_{1}^{*}}}+[g_{k}]_{1+\alpha,\overline{Q_{1}^{*}}}\right),
\end{eqnarray*}
since $0<\rho<1$. Having made this observation, we proceed to estimate each of the aforementioned terms:
\begin{itemize}
\item[(I)] {\bf \text{Estimate for} $[g_{k}]_{1+\alpha,\overline{Q_{1}^{*}}}$}.\\
Note that $DP_{k}(x,t)=\mathrm{M}_{k}x+a_{k}$ and therefore it does not depend on the time variable $t$. Hence,
\begin{eqnarray*}
[g_{k}]_{1+\alpha,\overline{Q_{1}^{*}}}&\leq& \rho^{k(\alpha-\gamma)}\left([g]_{1+\alpha,\overline{Q_{\rho^{k}}^{*}}}+[\beta]_{1+\alpha,\overline{Q_{\rho^{k}}^{*}}}(\rho^{k}\|\mathrm{M}_{k}\|+|a_{k}|)\right)\\
&\leq&[g]_{1+\alpha,\overline{Q_{\rho^{1}}^{*}}}+\mathrm{C}(\mathfrak{C},\gamma,\rho)[\beta]_{1+\alpha,\overline{Q_{\rho^{1}}^{*}}},
\end{eqnarray*}
since by the item (d) of the induction hypothesis,
\begin{eqnarray}
|a_{k}|&\leq& \sum_{j=1}^{k}|a_{j}-a_{j-1}|\leq \mathfrak{C}\sum_{j=1}^{\infty}\rho^{(1+\gamma)(j-1)}=\frac{\mathfrak{C}}{1-\rho^{1+\gamma}},\label{6.10}\\
\|\mathrm{M}_{k}\|&\leq& \sum_{j=1}^{k}\|\mathrm{M}_{j}-\mathrm{M}_{j-1}\|\leq \sum_{j=1}^{\infty}\rho^{\gamma(j-1)}=\frac{\mathfrak{C}}{1-\rho^{\gamma}}.\label{6.11}
\end{eqnarray}
In this case, $\rho^{k}\|\mathrm{M}_{k}\|=o(k)$ as $k\to \infty$, consequently, bounded.
\item[(II)] {\bf \text{Estimate for} $[Dg_{k}]_{\alpha,\overline{Q_{1}^{*}}}$}.\\
For each $1\leq i\leq n-1$ we have that 
\begin{eqnarray*}
[D_{i}g_{k}]_{\alpha,\overline{Q_{1}^{*}}}&\leq&\rho^{k(\alpha-\gamma)}\left([D_{i}g]_{\alpha,\overline{Q_{\rho^{k}}^{*}}}+[D_{i}\beta]_{\alpha,\overline{Q_{\rho^{k}}^{*}}}(\rho^{k}\|\mathrm{M}_{k}\|+|a_{k}|)\right)\\
&+&2^{1-\alpha}\|D_{i}\beta\|_{L^{\infty}(\overline{Q_{\rho^{k}}^{*}})}\rho^{k(1-\gamma)}\|\mathrm{M}_{k}\|+\rho^{k(\alpha-\gamma)}[\beta]_{\alpha,\overline{Q_{\rho^{k}}^{*}}}\mathrm{C}(n)\|\mathrm{M}_{k}\|\\
&+&2^{1-\alpha}\|\beta\|_{L^{\infty}(\overline{Q_{\rho^{k}}^{*}})}\rho^{k(1-\gamma)}\mathrm{C}(n)\|\mathrm{M}_{k}\|\\
&\leq& \mathrm{C}(n,\rho,\gamma, \mathfrak{C},\|\beta\|_{L^{\infty}(\overline{Q_{1}^{*}})},[\beta]_{\alpha,\overline{Q_{1}^{*}}}, [D_{i}\beta]_{\alpha,\overline{Q_{1}^{*}}} [D_{i}g]_{\alpha,\overline{Q_{1}^{*}}}),
\end{eqnarray*}
where we used the estimates \eqref{6.10}--\eqref{6.11} above, as well as the regularity of the boundary data $\beta$ and $g$. In particular, this yields the desired bound for $Dg_{k}$.
\end{itemize}
Thus, we conclude that the $C^{1+\alpha,\frac{1+\alpha}{2}}$-norm of $g_{k}$ is bounded independently of $k$. Therefore, we may invoke Lemma \ref{Lemma5.2} and conclude the existence of a quadratic polynomial
$$
\bar{P}(x,t)=\frac{1}{2}x^{T}\bar{\mathrm{M}}x+\bar{a}\cdot x+\bar{b}t+\bar{c}
$$
such that 
\begin{equation}\label{6.12}
\sup_{\overline{Q_{\rho}^{+}}}|u_{k}-\bar{P}|\leq \rho^{2+\gamma}.
\end{equation}
Moreover, the coefficients satisfy:
\begin{itemize}
\item[(i)] $F_{k}(\bar{\mathrm{M}},0,0)=\bar{b}$
\item[(ii)] $\beta_{k}(0,0)\cdot \bar{a}=g_{k}(0.0)$
\item[(iii)] For each $1\leq i\leq n-1$
\begin{equation*}
\sum_{j=1}^{n}(D_{i}(\beta_{k})_{j}(0,0)\bar{a}_{j}+(\beta_{k})_{j}(0,0)\bar{\mathrm{M}}_{ij})=D_{i} g_{k}(0,0).
\end{equation*}
\item[(iv)] $|\bar{a}|+|\bar{b}|+|\bar{c}|+\|\bar{\mathrm{M}}\|\leq \mathfrak{C}$
\end{itemize}
Now we define
\[
P_{k+1}(x,t)=P_{k}(x,t)+\rho^{k(2+\gamma)}\bar{P}(\rho^{-k}x,\rho^{-2k}t).
\]
In this case, by rescaling \eqref{6.12} we obtain estimate \eqref{6.7} for $k+1$.
Moreover, using the items (a)-(d) (from the induction hypothesis) and (i)-(iv) it follows the properties for the quadratic profile $P_{k+1}$. This proves the claim by induction. Now, from the properties of this sequence, we see that the sequences of coefficients
$(c_{k})_{k}$, $(b_{k})_{k}$, $(a_{k})_{k}$, and $(\mathrm{M}_{k})_{k}$ are Cauchy sequences and converge, respectively, to $u(0,0)$, $u_{t}(0,0)$, $Du(0,0)$, and $D^{2}u(0,0)$, with the following rate of convergence:
\begin{eqnarray}\label{6.13}
|c_{k}-u(0,0)|\leq \frac{\mathfrak{C}\rho^{k(2+\gamma)}}{1-\rho^{2+\gamma}}, \quad|b_{k}-u_{t}(0,0)|\leq \frac{\mathfrak{C}\rho^{k\gamma}}{1-\rho^{\gamma}}, \quad|a_{k}-Du(0,0)|\leq \frac{\mathfrak{C}\rho^{k(1+\gamma)}}{1-\rho^{1+\gamma}}
\end{eqnarray}
and
\begin{equation}\label{6.14}
\|\mathrm{M}_{k}-D^{2}u(0,0)\|\leq \frac{\mathfrak{C}\rho^{k\gamma}}{1-\rho^{\gamma}}.
\end{equation}
Now, let
\begin{equation*}
P_{u}(x,t)=\frac{1}{2}x^{T}D^{2}u(0,0)x+Du(0,0)\cdot x+u_{t}(0,0)t+u(0,0).
\end{equation*}
Given $0<r<\rho$, choose $k\in\mathbb{N}$ such that
$\rho^{k+1}<r\leq \rho^{k}$.
Using the coefficient estimates \eqref{6.13}--\eqref{6.14}, together with the approximation control \eqref{6.7}, it follows that
\begin{eqnarray*}
\sup_{(x,t)\in \overline{Q_{r}^{+}}}|u(x,t)-P_{u}(x,t)|&\leq& \sup_{(x,t)\in \overline{Q_{\rho^{k}}^{+}}}|u(x,t)-P_{k}(x,t)|+\sup_{(x,t)\in \overline{Q_{\rho^{k}}^{+}}}|P_{k}(x,t)-P_{u}(x,t)|\\
&\leq&\rho^{k(2+\gamma)}+\mathfrak{C}\rho^{k(2+\gamma)}\left(\frac{2}{1-\rho^{\gamma}}+\frac{1}{1-\rho^{1+\gamma}}+\frac{1}{1-\rho^{2+\gamma}}\right)\\
&\leq&\mathfrak{C}_{0}r^{2+\gamma},
\end{eqnarray*}
where
\begin{equation*}
\mathfrak{C}_{0}=\frac{1}{\rho^{2+\gamma}}\left(1+\left(\frac{2}{1-\rho^{\gamma}}+\frac{1}{1-\rho^{1+\gamma}}+\frac{1}{1-\rho^{2+\gamma}}\right)\right).
\end{equation*}
This proves the desired estimate on the flat boundary of the cylinder. For interior points, we invoke \cite[Theorem 1.1]{WangII} and thus complete the proof.
\end{proof}

\begin{remark}
It seems plausible that the techniques developed to derive the estimates in Theorem \ref{Theorem5.3} can be adapted
to obtain regularity results in the setting of more general moduli of continuity, such as the Dini case.
For basic properties, characterizations of these spaces, and interior estimates in this framework, we refer the reader to \cite{Bao2002,BSO26,daSdosP2019,ZC2002}.
\end{remark}

\section{The second borderline case: Log-Lipschitz regularity of the gradient}\label{Section6}

In this section, we will deal with the limiting integrability case, i.e., when the source term $f$ has \textit{bounded mean oscillation} (see the Definition \ref{DefBOM}). In contrast to the preceding sections, we begin by examining the case in which the governing operator is independent of the coefficients. More specifically, we focus on solutions to the following problem
\begin{equation} \label{6.1}
\left\{
\begin{array}{rclcl}
F(D^{2}u)-u_{t} & = & f(x,t) & \text{in} & Q^{+}_{1}, \\
\beta(x,t)\cdot Du(x,t) & = & g(x,t) & \text{on} & Q^{*}_{1}.
\end{array}
\right.
\end{equation}
where we will assume the following regularity assumption: 
\begin{itemize}
\item[{$(\mathbf A^{\star})$}] For each pair $(\mathrm{M}
,K)\in \text{Sym}(n)\times \mathbb{R}$ such that $F(\mathrm{M})=K$, the auxiliary problem
\[
\left\{
\begin{array}{rclcl}
F(D^{2}{h}+\mathrm{M})-{h}_{t} & = & K & \text{in} & Q^{+}_{\frac{7}{8}}, \\
\beta\!\cdot\! D{h} & = & g(x,t) & \text{on} & Q^{*}_{\frac{7}{8}}.
\end{array}
\right.
\]
fulfills that $C^{2+\alpha,\frac{2+\alpha}{2}}$-estimates, that is, $h\in C^{2+\alpha,\frac{2+\alpha}{2}}(\overline{Q^{+}_{\frac{2}{3}}}) 
\cap C^{0}(Q^{+}_{\frac{7}{8}}\cup Q^{*}_{\frac{7}{8}})$
whenever $\beta,,g\in C^{1+\alpha,\frac{1+\alpha}{2}}(\overline{Q^{*}_{\frac{7}{8}}})$, and that the following estimate holds
\[
\|{h}\|_{C^{2+\alpha,\frac{2+\alpha}{2}}(\overline{Q^{+}_{\frac{2}{3}}})}
\leq \mathrm{C}_{\star\star}\!
\left(\|{h}\|_{L^{\infty}(Q^{+}_{\frac{7}{8}})}
+\|g\|_{C^{1+\alpha,\frac{1+\alpha}{2}}(\overline{Q^{*}_{\frac{7}{8}}})}\right),
\]
for a universal constant $\mathrm{C}_{\star\star}>0$.
\end{itemize}
The class of operators $F$ satisfying condition ${(\mathbf{A^{\star}})}$ is nonempty, as illustrated in Remark \ref{remark5.1}.

As emphasized in the introduction, the idea behind proving the parabolic $C^{1,\mathrm{Log\text{-}Lip}}$ estimates in this setting is to employ a second-degree polynomial approximation with a decay rate of order two for the perturbed problem associated with a matrix $\mathrm{\tilde{M}}$ satisfying $F(\mathrm{\tilde{M}})=(f)_{Q_{1}^{+}}$. This is made possible by the following key lemma.

\begin{lemma}\label{Lemma6.1}
Let $u$ be a normalized viscosity solution of 
\begin{equation*}
\left\{
\begin{array}{rclcl}
F(D^{2}u+\tilde{\mathrm{M}})-u_t& = & f(x,t) & \text{in} & Q^{+}_{1}, \\
\beta\cdot Du & = & g(x,t) & \text{on} & Q^{*}_{1},
\end{array}
\right.
\end{equation*}
where $\beta, g\in C^{1+\alpha,\frac{1+\alpha}{2}}(\overline{Q_1^{*}})$ for some constant $\alpha\in(0,1)$, $\tilde{\mathrm{M}}\in \text{Sym}(n)$ is such that $F(\tilde{\mathrm{M}})=(f)_{Q_{1}^{+}}$ and assume that the hypothesis ${(\mathbf{A^{\star}})}$ holds. There exist $\delta>0$ and $\rho\in(0,\frac{1}{2}]$ depending only on $n$, $\lambda$, $\Lambda$, $\mu_{0}$, $\mathrm{C}_{\star\star}$, $\|\beta\|_{C^{1,\alpha}(\overline{Q^*_{1}})}$, and $\|g\|_{C^{1,\alpha}(\overline{Q^*_{1}})}$ such that, if 
\begin{equation*}
[f]_{\mathrm{BMO}(Q_{1}^{+})}\le\delta.
\end{equation*}
Then, there exists a quadratic polynomial $P(x,t)=\frac{1}{2}x^{t}\mathrm{M}x+a\cdot x+bt+c$ with universally bounded coefficients, in the following sense 
\begin{equation*}
|a|+|b|+|c|+\|\mathrm{M}\|\le \mathfrak{C}(n,\lambda,\Lambda,\mu_{0},\alpha,\|\beta\|_{C^{1+\alpha,\frac{1+\alpha}{2}}(\overline{Q^*_{1}})},\|g\|_{C^{1+\alpha,\frac{1+\alpha}{2}}(\overline{Q^*_{1}})}),
\end{equation*}
such that 
\begin{equation*}
\sup_{\overline{Q_{\rho}^{+}}}|u-P|\le\rho^{2}.
\end{equation*}
Moreover, the coefficients $a= ({a}_1, \dots, {a}_n)$, ${b}$, and $\mathrm{M} = (\mathrm{M}_{ij})_{n \times n}$ satisfy
$$
F(\mathrm{M} + \tilde{\mathrm{M}}) -b= (f)_1, \quad \beta(0,0) \cdot {a} = g(0,0),
$$
and
\begin{equation}\label{eq:6.3}
 \sum_{j=1}^{n} (D_i \beta_j(0,0){a}_j + \beta_j(0,0)\mathrm{M}_{ij})  = D_i g(0,0), \quad \text{for all } 1 \leq i \leq n-1. 
\end{equation}
Here, $\beta_1, \dots, \beta_n$ denote the coordinate functions of the vector field $\beta$.
\end{lemma}
\begin{proof}
First, we set $\varepsilon>0$, which we will choose later. By Lemma \ref{ApproxBMO}, there exists $\delta=\delta(n,\lambda,\Lambda,\mu_{0},\varepsilon)>0$ such that if $\|f\|_{\mathrm{BMO}(Q_{1}^{+})}\le\delta$, then the viscosity solution $h$ of
\begin{equation}\label{eq:6.5}
\left\{
\begin{array}{rclcl}
F(D^2 h+\mathrm{\tilde{M}}) - h_{t} &=& (f)_{Q_{1}^{+}} & \text{in} & Q^{+}_{\frac{7}{8}}, \\
\beta \cdot Dh&=& g(x,t) & \text{on} & Q_{\frac{7}{8}}^{*}, \\
h &=& u & \text{on} & \partial_{p}Q^{+}_{\frac{7}{8}}\setminus Q_{\frac{7}{8}}^{*}
\end{array}
\right.
\end{equation}
such that
\begin{equation}\label{eq:6.4}
\sup_{Q_{\frac{7}{8}}^{+}}|u-h|\le\varepsilon.
\end{equation}
Thus, combining hypothesis ${(\mathbf{A^{\star}})}$ with the condition $F(\mathrm{\tilde{M}})=(f)_{Q_{1}^{+}}$, we conclude that
\[
h\in C^{2+\alpha,\frac{2+\alpha}{2}}(\overline{Q_{\frac{2}{3}}^{+}})
\]
and the following estimate is valid
\[
\|{h}\|_{C^{2+\alpha,\frac{2+\alpha}{2}}(\overline{Q^{+}_{\frac{2}{3}}})}
\leq \mathrm{C}_{\star\star}\!
\left(\|{h}\|_{L^{\infty}(Q^{+}_{\frac{7}{8}})}
+\|g\|_{C^{1+\alpha,\frac{1+\alpha}{2}}(\overline{Q^{*}_{\frac{7}{8}}})}\right),
\]
for a universal constant $\mathrm{C}_{\star\star}>0$. Invoking the A.B.P.T. estimate \ref{ABPTestimates}, we obtain
\begin{equation}\label{eq:6.6}
\|{h}\|_{C^{2+\tilde{\alpha},\frac{2+\tilde{\alpha}}{2}}(\overline{Q}_{\frac{2}{3}}^{+})}\le \mathrm{C}=\mathrm{C}\left(n,\lambda,\Lambda,\mu_{0},\alpha,\|\beta\|_{C^{1+\alpha,\frac{1+\alpha}{2}}(\overline{Q^*_{1}})},\|g\|_{C^{1+\alpha,\frac{1+\alpha}{2}}(\overline{Q^*_{1}})}\right).
\end{equation}
Thus, defining $c=h(0,0)$, $b=h_t(0,0)$, $a=Dh(0,0)$, and $\mathrm{M}=D^{2}h(0,0)$, the quadratic polynomial
\[
P(x,t)=c+bt+a\cdot x+\frac{1}{2}x^{t}\mathrm{M}x
\]
is well defined. Moreover, it follows immediately that
\[
F(\mathrm{M}+\tilde{\mathrm{M}})=(f)_{Q_{1}^{+}}.
\]
As in the Lemma \ref{Lemma5.2}, by \eqref{eq:6.6}, it follows that
\begin{equation}\label{eq:6.7}
\sup_{Q_{r}^{+}}|h-P|\le \mathrm{C}_{2}r^{2+\alpha}.
\end{equation}
for a universal constant $\mathrm{C}_{2}>0$. In this point, we make the following universal choices of the constants
\begin{equation}\label{eq:6.8}
\rho\defeq\min\left\{\frac{1}{2},\left(\frac{1}{2\mathrm{C}_{2}}\right)^{\frac{1}{\alpha}}\right\} \quad \text{and} \quad \varepsilon\defeq\frac{1}{2}\rho^{2}.
\end{equation}
With these choices, the constant $\delta>0$ is determined by Lemma \ref{ApproxBMO}. Moreover, the universal bounds for the constants $a$, $b$, and $\mathrm{M}$ follow directly from \eqref{eq:6.6}. Finally, combining \eqref{eq:6.4}, \eqref{eq:6.8}, and \eqref{eq:6.7}, we obtain
$$
\sup_{\overline{Q_{\rho}^{+}}}|u-P|\le\rho^{2},
$$
thereby obtaining the desired estimate. For the remaining conditions on the coefficients $a$, $b$, and $\mathrm{M}$, the $C^{2,\alpha}$-regularity of the solution $h$ to \eqref{eq:6.5} yields the pointwise identity
\begin{equation}\label{eq:6.9}
\beta(x,t)\cdot D{h}(x,t)=g(x,t) \quad \text{on} \quad Q^*_{\frac{1}{2}}.
\end{equation}
Evaluating at $(0,0)$ gives $\beta(0,0)\cdot a=g(0,0)$. For any $1\le i<n$, differentiating \eqref{eq:6.9} in the $i$-th direction and evaluating at $(0,0)$ yields condition \eqref{eq:6.3}, completing the proof.
\end{proof}
Now, we are in a position to present the first main result of this section.

\begin{theorem}[{\bf Regularity $C^{1,\rm{Log-Lip}}$}]\label{Theorem6.2}
Let $u$ be a viscosity solution for \eqref{6.1}, where $\beta,g\in C^{1+\alpha,\frac{1+\alpha}{2}}(\overline{Q_{1}^{*}})$ and $f\in \mathrm{BMO}(Q_{1}^{+})$. Assume that the hyphotesis ${(\mathbf{A^{\star})}}$ holds. Then, $u \in C^{1,\mathrm{Log\text{-}Lip}}\left(\overline{{Q}_{\frac{1}{2}}^+}\right)$. Moreover, the following estimate holds
\begin{eqnarray*}
\sup_{\substack{x, y \in \overline{{Q}_{\frac{1}{2}}^+} \\ x \neq y}} \frac{|u(x,t) - u(y,t) - Du(y,t) \cdot ((x,t) - (y,t))|}{(d_{\text{par}}((x,t),(y,t)))^2 \ln|(d_{\text{par}}((x,t),(y,t))|^{-1}} \leq \mathrm{C} \left(\|u\|_{L^\infty({Q}_1^+)} + [f]_{\mathrm{BMO}({Q}_1^+)}+\|g\|_{C^{1+\alpha,\frac{1+\alpha}{2}}(\overline{Q^*_1})}\right),
\end{eqnarray*}
where $\mathrm{C} > 0$ is a constant depending only on $n$, $\lambda$, $\Lambda$, $\mu_0$, $\alpha$, $\mathrm{C}_{\star\star}$, and  $\|\beta\|_{C^{1+\alpha,\frac{1+\alpha}{2}}(\overline{{Q^*_1}})}$.
\end{theorem}
\begin{proof}
Up to scaling, we may assume $g(0)=0$, $Dg(0)=0$, and 
$$\|u\|_{L^{\infty}(Q_{1}^{+})}\le1, \quad [f]_{\mathrm{BMO}(Q_{1}^{+})}\le\delta \quad \text{and} \quad \|g\|_{C^{1,\alpha}(\overline{Q^*_{1}})}\le1,$$
where $\delta>0$ is the constant from Lemma \ref{Lemma6.1}. To establish regularity at points $(x,t)\in \overline{Q^*_{\frac{1}{2}}}$, a localization and translation argument reduces the problem to the origin $(0,0)$. We claim there exists a sequence of quadratic polynomials $({P}_{k})_{k\ge-1}$ of the form
$$P_{k}(x,t)
=\frac{1}{2}x^{T}\mathrm{M}_{k}x+ a_{k}\cdot x + b_{k}t + c_{k},$$ 
satisfying:
\begin{itemize}
    \item[(i)] $F(\mathrm{M}_{k})-b_{k}=(f)_{Q_{1}^{+}}$;
    \item[(ii)] $\displaystyle\sup_{\overline{Q_{\rho^{k}}^{+}}}|u-P_{k}|\le\rho^{2k}$;
    \item[(iii)] $|c_{k-1}-c_{k}|+\rho^{2(k-1)}|b_{k-1}-b_{k}|+\rho^{k-1}|a_{k-1}-a_{k}|+\rho^{2(k-1)}\|\mathrm{M}_{k-1}-\mathrm{M}_{k}\|\le \mathfrak{C}\rho^{2(k-1)}$;
    \item[(iv)] $\beta(0,0)\cdot a_{k}=0$;
    \item[(v)] For each $1\le i\le n-1$, 
    $$\sum_{j=1}^{n}(D_{i}\beta_{j}(0,0)(a_{k})_{j}+\beta_{j}(0,0)(\mathrm{M}_{k})_{ij})=0,$$ 
    for all $k\ge0$, where $\rho$ and $\mathfrak{C}$ are, respectively, the radius and the universal boundedness constant from Lemma \ref{Lemma6.1}, $(a_{k})_{1},...,(a_{k})_{n}$ denote the components of $a_{k}$, and $(\mathrm{M}_{k})_{ij}$ the entries of $\mathrm{M}_{k}$.
\end{itemize}
For the base case $k=0$, select a symmetric matrix $\mathrm{M}_{0}$ solving the linear system in (v) with $F(\mathrm{M}_{0})=(f)_{Q_{1}^{+}}$. Taking ${P}_{0}={P}_{-1}=\frac{1}{2}x^{T}\mathrm{M}_{0}x$ then establishes the case $k=0$. Now, we assume that the statement holds for some $k$, and we define the following auxiliary function 
$$u_{k}(x):=\frac{(u-{P}_{k})(\rho^{k}x, \rho^{2k}t)}{\rho^{2k}}, \quad x\in Q_{1}^{+}\cup Q^*_{1}.$$
which is a viscosity solution to
\begin{equation*}
\left\{
\begin{array}{rclcl}
F(D^{2}v_{k}+\mathrm{M}_{k})-(v_k)_t& = & f_{k}(x,t) & \text{in} & Q^{+}_{1}, \\
\beta_{k}\cdot Du_{k} & = & g_{k}(x,t) & \text{on} & Q^{*}_{1},
\end{array}
\right.
\end{equation*}
where
$$
\begin{cases}f_{k}(x,t):=f(\rho^{k}x,\rho^{2k}t)\\ 
\beta_{k}(x,t):=\beta(\rho^{k}x,\rho^{2k}t)\\ 
g_{k}(x,t):=\rho^{-k}(g(\rho^{k}x,\rho^{2k}t)-\beta(\rho^{k}x,\rho^{2k}t)\cdot D{P}_{k}(\rho^{k}x,\rho^{2k}t))
\end{cases}.
$$
Now, note that, by the induction hypothesis, it follows from (ii) that $\|u_{k}\|_{L^{\infty}(Q_{1}^{+})}\le1$. Moreover, by the definition of $f_{k}$ we have 
\begin{eqnarray*}
    [{f}_{k}]_{\mathrm{BMO}(Q_{1}^{+})}&=&\sup_{(x_{0},t_{0})\in Q_1^+,r>0}\left(\intav{{Q}_{r}(x_{0},t_{0})\cap{Q}_{1}^{+}}|f_{k}(x,t)-({f})_{(x_{0},t_{0}),r}|\,dxdt\right)\\
    &=&\sup_{(x_{0},t_{0})\in Q_{1}^{+},r>0}\left(\intav{Q_{\rho^{k}r}(\rho^{k}x_{0},\rho^{2k}t_{0})\cap Q_{1}^{+}}|f(z,s)-(f)_{(\rho^{k}x_{0},\rho^{2k}t_{0}),r\rho^{k}}|\,dzds\right).
\end{eqnarray*}
Consequently,
$$[f_{k}]_{\mathrm{BMO}(Q_{1}^{+})}\le[f]_{\mathrm{BMO}(Q_{1}^{+})}\le\delta.$$

Additionally, we observe that $\beta_{k}, \in C^{1+\alpha,\frac{1+\alpha}{2}}(\overline{Q^*_{1}})$, since $\beta\in C^{1+\alpha,\frac{1+\alpha}{2}}(\overline{Q^*_{1}})$ and $\rho\in(0,\frac{1}{2}]$. On the other hand, for each $i=1,...,n-1$, item (v) of the induction hypothesis yields 
\begin{eqnarray*}
|D_{i}g_{k}(x,t)|&\le&|D_{i}g(\rho^{k}x, \rho^{2k}t)-D_{i}g(0,0)|\\
&+&\left|\sum_{j=1}^{n}[(D_{i}\beta_{j}((\rho^{k}x, \rho^{2k}t)-D_{i}\beta_{j}(0,0))(a_{k})_{j}+D_{i}\beta_{j}(0,0)\rho^{k}((\mathrm{M}_{k})_{ij}x_{j})]\right|\\
&+&\left|\sum_{j=1}^{n}[(\beta_{j}((\rho^{k}x, \rho^{2k}t)-\beta_{j}(0,0))(\mathrm{M}_{k})_{ij}]\right|\\
&\le&[D_{i}g]_{\alpha,\overline{Q^*_{1}}}+[D_{i}\beta]_{\alpha,\overline{Q^*_{1}}}\mathrm{C}(n)(\rho^{k\alpha}|a_{k}|+\rho^{k\alpha}\|\mathrm{M}_{k}\|)+\mathrm{C}(n)\|D_{i}\beta\|_{L^{\infty}(\overline{Q^*_{1}})}\rho^{k}\|\mathrm{M}_{k}\|
\end{eqnarray*}
for all$ (x,t)\in \overline{Q^*_{1}}$.

From item (iii) of the induction hypothesis, 
\begin{equation}\label{eq:6.12}
\rho^{k\alpha}\|\mathrm{M}_{k}\|=o(1), \quad \rho^{k\alpha}|a_{k}|=o(1), \quad \text{and} \quad \rho^{k}|b_{k}|=o(1) \text{ as } k\rightarrow\infty,
\end{equation} 
which together with $\beta, g\in C^{1+\alpha,\frac{1+\alpha}{2}}(\overline{Q^*_{1}})$ implies $\|D_{i}g_{k}\|_{L^{\infty}(\overline{Q^*_{1}})}\le \mathrm{C}$ for some constant $\mathrm{C}>0$ independent of $k$. Thus $Dg_{k}$ is uniformly bounded in $L^{\infty}$. 

As in the Section \ref{Section5}, $DP_k=\mathrm{M}_kx+a_k$ then
\begin{eqnarray*}
[g_{k}]_{1+\alpha,\overline{Q_{1}^{*}}}&\leq& \rho^{k\alpha}\left([g]_{1+\alpha,\overline{Q_{\rho^{k}}^{*}}}+[\beta]_{1+\alpha,\overline{Q_{\rho^{k}}^{*}}}(\rho^{k}\|\mathrm{M}_{k}\|+|a_{k}|)\right)\\
&\leq&[g]_{1+\alpha,\overline{Q_{\rho^{}}^{*}}}+\mathrm{C}(\mathfrak{C},\rho)[\beta]_{1+\alpha,\overline{Q_{\rho^{}}^{*}}},
\end{eqnarray*}
and
\begin{eqnarray}
[D_{i}g_{k}]_{\alpha,\overline{Q^*_{1}}}&\le&[D_{i}g]_{\alpha,\overline{Q^*_{\rho^{k}}}}+[D_{i}\beta(\rho^{k}\cdot, \rho^{2k}\cdot)\cdot D{P}_{k}(\rho^{k}\cdot, \rho^{2k}\cdot)]_{\alpha,\overline{Q^*_{1}}}\nonumber\\
        &&+[\beta(\rho^{k}\cdot, \rho^{2k}\cdot)\cdot D_{i}(D{P}_{k}(\rho^{k}\cdot, \rho^{2k}\cdot))]_{\alpha,\overline{Q^*_{1}}}\label{eq:6.13}
\end{eqnarray}
Next, we will analyze each term in the right-hand side of \eqref{eq:6.13} separately. 

In effect, by $g\in C^{1+\alpha,\frac{1+\alpha}{2}}(\overline{Q^*_{1}})$, it follows that 
$$[D_{i}g]_{\alpha,\overline{Q^*_{\rho^{k}}}}\le\|g\|_{C^{1+\alpha,\frac{1+\alpha}{2}}(\overline{Q^*_{1}})}<\infty.$$ 
Now, condition \eqref{eq:6.12} and the preceding remark yield
\begin{eqnarray*}
[D_{i}\beta(\rho^{k}\cdot, \rho^{2k}\cdot)\cdot D{P}_{k}(\rho^{k}\cdot, \rho^{2k}\cdot)]_{\alpha,\overline{Q^*_{1}}}&\le&[D_{i}\beta]_{\alpha,\overline{Q^*_{\rho^{k}}}}(\rho^{k\alpha}\|D{P}_{k}\|_{L^{\infty}(\overline{Q^*_{\rho^{k}}})})+2\|D_{i}\beta\|_{L^{\infty}(\overline{Q^*_{\rho^{k}}})}\rho^{k}\|\mathrm{M}_{k}\|\\
    &\le&[D_{i}\beta]_{\alpha,\overline{Q^*_{1}}}(\rho^{k\alpha}\|D{P}_{k}\|_{L^{\infty}(\overline{Q^*_{\rho^{k}}})})\\
    &+&2\|D_{i}\beta\|_{L^{\infty}(\overline{Q^*_{1}})}\rho^{k}(\mathfrak{C}k+\|\mathrm{M}_{0}\|)<\infty
\end{eqnarray*}
since, by induction hypothesis, see item (iii),
$$\|D{P}_{k}\|_{L^{\infty}(\overline{Q^*_{\rho^{k}}})}\le\frac{1}{1-\rho}\mathfrak{C}+\mathfrak{C}\cdot o(1) \text{ as } k\rightarrow\infty.$$

Similarly, we obtain 
$$[\beta(\rho^{k}\cdot, \rho^{2k}\cdot)\cdot D_{i}(D{P}_{k}(\rho^{k}\cdot, \rho^{2k}\cdot))]_{\alpha,\overline{Q^*_{1}}}\le\rho^{k\alpha}[\beta]_{\alpha,\overline{Q^*_{1}}}\|\mathrm{M}_{k}\|+\mathrm{C}(n)\|\beta\|_{L^{\infty}(\overline{Q^*_{1}})}\rho^{k}\|\mathrm{M}_{k}\|=o(1), \text{ as } k\rightarrow\infty,$$ 
where again we use item (iii).

Now, combining the estimates for $\|Dg_{k}\|_{L^{\infty}(\overline{Q^*_{1}})}$ and $[Dg_{k}]_{\alpha,\overline{Q^*_{1}}}$ with continuity and the Mean Value Inequality yields $g_{k}\in C^{1+\alpha,\frac{1+\alpha}{2}}(\overline{Q^*_{1}})$. We are therefore in the setting of Lemma \ref{Lemma6.1}, which provides a quadratic polynomial $\tilde{P}(x)=\tilde{c}+\tilde{b}t+\tilde{a}\cdot x+\frac{1}{2}x^{T}\tilde{\mathrm{M}}x$ satisfying:
\begin{itemize}
    \item[(a)] $|\tilde{c}|+|\tilde{b}|+|\tilde{a}|+\|\tilde{\mathrm{M}}\|\le \mathfrak{C}$;
    \item[(b)] $\beta_{k}(0,0)\cdot\tilde{a}=g_{k}(0,0)$;
    \item[(c)] For each $1\le i<n$, 
    $$\sum_{j=1}^{n}(D_{i}(\beta_{k})_{j}(0,0)\tilde{a}_{j}+(\beta_{k})_{j}(0,0)\tilde{\mathrm{M}}_{ij});$$
    \item[(d)] $F(\tilde{\mathrm{M}}+\mathrm{M}_{k})-\tilde{b}=(f)_{Q_{1}^{+}}$;
    \item[(e)] $\displaystyle\sup_{\overline{Q_{\rho}^{+}}}|u_{k}-\tilde{{P}}|\le\rho^{2}$.
\end{itemize}

Define 
$$a_{k+1}=a_{k}+\rho^{2k}\tilde{a}, \quad b_{k+1}=b_{k}+\rho^{k}\tilde{b}\quad c_{k+1}=c_k+\rho^{2k}\tilde{c} \quad \text{and} \quad \mathrm{M}_{k+1}=\mathrm{M}_{k}+\tilde{\mathrm{M}},$$
and set
$${P}_{k+1}(x):=c_{k+1}+b_kt+a_{k+1}\cdot x+\frac{1}{2}x^{t}\mathrm{M}_{k+1}x.$$ 
Then (e) implies
$$\sup_{Q_{\rho^{k+1}}^{+}}|u-{P}_{k+1}|\le\rho^{2(k+1)},$$
establishing condition (ii) for $k+1$. Condition (i) follows from Lemma \ref{Lemma6.1}. Moreover,
$$
|c_k-c_{k+1}|+\rho^{2k}|b_{k}-b_{k+1}|+\rho^{k}|a_{k}-a_{k+1}|+\rho^{2k}\|\mathrm{M}_{k}-\mathrm{M}_{k+1}\|\le\rho^{2k}(|\tilde{c}|+|\tilde{b}|+|\tilde{a}|+\|\tilde{\mathrm{M}}\|)\le \mathfrak{C}\rho^{2k},$$ 
which verifies condition (iii) for $k+1$. Conditions (b), (c) and the induction hypotheses (iv)-(v) imply the corresponding statements for $k+1$. This completes the induction process.

Now, note that condition (iii) ensures that the sequences $(a_{k})_{k\in\mathbb{N}}$ and $(c_{k})_{k\in\mathbb{N}}$ are Cauchy. Thus, we may consider 
$$a_{\infty}=\lim_{k\rightarrow\infty}a_{k} \quad \text{and} \quad c_{\infty}=\lim_{k\rightarrow\infty}c_{k}.$$ 
Moreover, it is easy to see $c_{\infty}=u(0,0)$. On the other hand, from condition (iii), we have the following rate of convergence of the sequences $(a_{k})$ and $(c_{k})$:
\begin{equation}\label{eq:6.14}
    |u(0,0)-c_{k}|\le\frac{\mathfrak{C}}{1-\rho^{2}}\rho^{2k} \quad \text{and} \quad |Du(0,0)-a_{k}|\le\frac{\mathfrak{C}}{1-\rho}\rho^{k}
\end{equation}
$$$$
for all $k\in\mathbb{N}$. Furthermore, although we have no guarantee of convergence of the sequences $(b_k)_{k\in \mathbb{N}}$ and $(\mathrm{M}_{k})_{k\in\mathbb{N}}$, observe that condition (iii) still ensures that
\begin{equation}\label{eq:6.15}
    \max\{|b_k|,\,\|\mathrm{M}_{k}\|\}\le \mathfrak{C}k
\end{equation}
$$$$ 
for all $k\in\mathbb{N}$.

Therefore, fixing $r\in(0,\rho)$ (thus $\rho\le1/2<\sqrt{1/e}$), we can choose $k\in\mathbb{N}$ in such a way that $\rho^{k+1}<r\le\rho^{k}$. Thus, by (ii), \eqref{eq:6.14} and \eqref{eq:6.15}, we get that 
\begin{eqnarray*}
    \sup_{Q_{r}^{+}}|u(x,t)-u(0,0)-a_{\infty}\cdot (x,t)|&\le&\sup_{Q_{r}^{+}}|u-P_{k}|+|u(0,0)-c_{k}|+\sup_{x\in Q_{r}^{+}}|(a_{k}-a_{\infty})\cdot (x,t)|\\
    &&+\sup_{Q^+_r}|b_k t|+\sup_{x\in Q_{r}^{+}}|\mathrm{M}_{k}x\cdot x|\\
    &\le&\sup_{Q_{\rho^{k}}^{+}}|u-{P}_{k}|+\frac{\mathfrak{C}}{1-\rho^{2}}\rho^{2k}+|a_{k}-a_{\infty}|r+(|b_k|+\|\mathrm{M}_{k}\|)r^{2}\\
    &\le&\rho^{2k}+\frac{\mathfrak{C}}{1-\rho^{2}}\rho^{2k}+\frac{\mathfrak{C}}{1-\rho}\rho^{k}r+2\mathfrak{C}kr^{2}.
\end{eqnarray*}

Finally, by using $r\le\rho^{k}$, it follows that
\begin{eqnarray*}
\sup_{Q_{r}^{+}}|u(x,t)-u(0,0)-a_{\infty}\cdot (x,t)|&\le&\rho^{2k}+\frac{\mathfrak{C}}{1-\rho^{2}}\rho^{2k}+\frac{\mathfrak{C}}{1-\rho}\rho^{2k}+2\mathfrak{C}k\rho^{2k}\\
&\le&\left(1+\frac{\mathfrak{C}}{1-\rho^{2}}+\frac{\mathfrak{C}}{1-\rho}\right)\rho^{2k}+2\mathfrak{C}k\rho^{2k}\\
&\le& \mathrm{C}(\rho^{2k}+k\rho^{2k})=\frac{\mathrm{C}}{\rho^{2}}\left(\frac{1}{k}+1\right)k\rho^{2(k+1)}\\
&\le& \mathrm{C}k\rho^{2(k+1)}\le-\mathrm{C}r^{2}\ln(r).
\end{eqnarray*} 
This establishes the initial statement. By a localization and translation argument, for every $(x,t)\in \overline{Q^*_{\frac{1}{2}}}$ and $0<r<\rho$, one has
\begin{equation}\label{eq:6.16}
\sup_{(y,s)\in Q_{r}^{+}(x,t)}|u(y,s)-u(x,t)-a_{\infty}(x,t)\cdot((y,s)-(x,t))|\le \mathrm{C}r^{2}|\ln r^{-1}|.
\end{equation} 
By the definition of differentiability, condition \eqref{eq:6.16} implies that $a_{\infty}(x,t)=Du(x,t)$, and therefore, combining with interior estimates \cite[Theorem 6.1]{DaST17}, $u\in\mathrm{par}- C^{1,\mathrm{Log\text{-}Lip}}(\overline{Q_{\frac{1}{2}}^{+}})$ with the desired estimate.
\end{proof}

Having completed this analysis for operators with constant coefficients, we now obtain the corresponding regularity result for viscosity solutions to \eqref{1.1} under suitable assumptions on the coefficients. This is summarized in the following theorem.
\begin{theorem}[{\bf Regularity $C^{1,\rm{Log-Lip}}$-General Case}]\label{Theorem6.3}
Let $u$ be a viscosity solution to \eqref{6.1}, where $\beta,g\in C^{1+\alpha,\frac{1+\alpha}{2}}(\overline{Q_{1}^{*}})$, $ f \in \mathrm{BMO}(Q^+_1)$ and ${(\mathbf{A^{\star}})}$ holds.
Assume that the oscillation of coefficients $\tilde{\Phi}_{F}$ satisfies 
$$
\displaystyle\sup_{(y,s) \in Q_1} \|\tilde{\Phi}_F((\cdot,\cdot);(y,s))\|_{C^{\alpha,\frac{\alpha}{2}}} := \mathcal{K} < \infty.
$$
Then, there exists $\mathrm{C} > 0$ that depends only on $C^{2+\alpha,\frac{2+\alpha}{2}}$ a priori regularity estimates available for 
\begin{equation} \label{regularity}
\left\{
\begin{array}{rclcl}
F(D^{2}u)-u_t& = & 0 & \text{in} & Q^{+}_{1}, \\
\beta\cdot Du & = & g(x,t) & \text{on} & Q^{*}_{1},
\end{array}
\right.
\end{equation}
$\mathcal{K}$, and universal constants, such that the following estimate holds
\begin{eqnarray*}
\sup_{(x,t)\in Q^{+}_r(y,s)} |u(x,t) - [u(y,s) + D u(y,s) \cdot ((x,t)-(y,s))]| \leq \mathrm{C} \mathcal{D}(u,f,g)  r^2 \log r^{-1},
\end{eqnarray*}
where
\[
\mathcal{D}(u,f,g)\defeq \|u\|_{L^\infty(Q^+_1)} + [f]_{\mathrm{BMO}(Q^+_1)} +\|g\|_{C^{1+\alpha,\frac{1+\alpha}{2}}(\overline{Q^*_1})}.
\]
\end{theorem}
\begin{proof}
To establish this result, it is sufficient to demonstrate the asymptotic decay of the rescaled operator's oscillation in the $L^p$ norm. Revisiting the proof of Theorem \ref{Theorem6.2}, in the $k$-th iterative step, the normalized residual function is governed by the operator:
$$ F_k(X, x, t) := F(X + \mathrm{M}_k, \rho^k x, \rho^{2k} t) - F(\mathrm{M}_k, \rho^k x, \rho^{2k} t). $$
By similar estimates in Section \ref{Section5}, it's easy to see,
$$ \tilde{\Phi}_{F_k}(x,t) \le 2 \tilde{\Phi}_F((\rho^k x, \rho^{2k} t), (0,0)) (1 + \|\mathrm{M}_k\|). $$
As $\tilde{\Phi}_F$ is $C^{\alpha,\frac{\alpha}{2}}$ class, we get $\tilde{\Phi}_F \le \mathcal{K} \rho^{k\alpha}$. Additionally, we have $\|\mathrm{M}_k\| \le \mathfrak{C}k$. Consequently, the $L^p$ norm of the oscillation in the normalized half-cylinder strictly obeys:
$$ \left( \int_{Q_1^+} \tilde{\Phi}_{F_k}(x,t)^p \, dx dt \right)^{\frac{1}{p}} \le \mathrm{C} \rho^{k\alpha} (1 + \mathfrak{C}k) \le \mathrm{\tilde{C}} \rho^{k\alpha} k. $$
This decay estimate legitimizes the direct application of the Lemma \ref{ApproxBMO}. The collapse of the spatial heterogeneity allows us to approximate the solution uniformly by an auxiliary profile $h$ associated with the perfectly frozen operator at the origin, $F_k(X, 0, 0)$. In practice, this readapts the mechanics of Lemma \ref{6.1} directly.
\end{proof}

We conclude the manuscript by establishing, as a consequence of Theorem \ref{Theorem6.3} and in the spirit of Corollary \ref{Corollary3.3}, that viscosity solutions to \eqref{1.1} belong to $C^{1+\alpha,\frac{1+\alpha}{2}}$ for every $\alpha\in(0,1)$.

\begin{corollary}
Assume that the hypotheses of Theorem \ref{Theorem6.3} are satisfied. Then every viscosity solution to \eqref{1.1} belongs to $C^{1+\alpha,\frac{1+\alpha}{2}}(\overline{Q_{\frac12}^{+}})$ for every $\alpha\in(0,1)$.
\end{corollary}
\begin{proof}
By Theorem \ref{Theorem6.3}, $u \in \text{par-}C^{1,\mathrm{Log\text{-}Lip}}(\overline{Q_{1/2}^+})$, in particular, the partial derivatives $D_{i}u$, $i=1,\dots,n$ belong to $\text{par-}C^{0,\mathrm{Log\text{-}Lip}}(\overline{Q_{1/2}^+})$ and by the same computations from the Corollary \ref{Corollary3.3} it follows that 
\[
[D_{i}u]_{\alpha,\overline{Q_{\frac{1}{2}^{+}}}}<\infty, 
\]
for any $i=1,\dots,n$ and $\alpha\in (0,1)$.

On the other hand, evaluating the parabolic $C^{1,\mathrm{Log\text{-}Lip}}$ estimate at $x=y$, the linear gradient term vanishes. The parabolic distance becomes $d_{par} = |t-s|^{1/2}$, leaving:
$$ |u(x,t) - u(x,s)| \le \mathrm{C} |t-s| \log |t-s|^{-1}. $$
Dividing this by the temporal H\"older weight $|t-s|^{\frac{1+\alpha}{2}}$ gives:
$$ \frac{|u(x,t) - u(x,s)|}{|t-s|^{\frac{1+\alpha}{2}}} \le C |t-s|^{\frac{1-\alpha}{2}} \log |t-s|^{-1}. $$
Again, the function $t\to t^\alpha\log t^{-1}$ is bounded for $0<\alpha<1$.

\end{proof}

\section{Final conclusions: Compatibility of the functional spaces}\label{Section7}

The regularity theory established in this work demonstrates a fundamental compatibility between the intrinsic parabolic geometry of the equation and the functional space associated with the source term. More precisely, the parabolic scaling
\[
(x,t)\mapsto (\rho x,\rho^2t)
\]
induces the degree of homogeneity $\iota=n+2$,  which completely determines the natural hierarchy of function spaces associated with the problem. This phenomenon becomes apparent when one considers the blow-up scaling associated with problem \eqref{1.1}, 
\begin{equation}\label{7.1}
u_\rho(x,t)=:\frac{u(\rho x,\rho^2t)}{\rho},
\end{equation}
whose corresponding source term is given by $f_\rho(x,t)=\rho\,f(\rho x,\rho^2t)$. Examining the $L^{p}(Q_{1}^{+})$-norm, one obtains the following relation
\[
\|f_\rho\|_{L^p(Q_1^+)}
=
\rho^{1-\frac{\iota}{p}}
\|f\|_{L^p(Q_\rho^+)}\leq \rho^{1-\frac{\iota}{p}}
\|f\|_{L^p(Q_{1}^+)},
\]
which, for $\rho \ll 1$, remains uniform provided that $p\geq n+2$. This scaling argument justifies the analysis carried out above, ranging from the critical Log-Lipschitz case discussed in Section~\ref{Section3} to the gradient Hölder regularity regime in Section~\ref{Section4}.

In view of this analysis, and inspired by the work \cite{DaST17}, our previous discussion is also consistent within the framework of \textit{anisotropic Lebesgue spaces}. Recall that a function \(f\) belongs to the anisotropic Lebesgue space \(L^{p,q}(Q_{1}^{+})\), for \(1\leq p\leq q<+\infty\), if
\begin{equation*}
\|f\|_{L^{p,q}(Q_{1}^{+})}:=\left(\int_{-1}^{0}\left(\int_{B_{1}^{+}}|f(x,t)|^{p}\,dx\right)^{\frac{q}{p}}dt\right)^{\frac{1}{q}}<\infty
\end{equation*}
which is a Banach space endowed with the norm \(\|\cdot\|_{L^{p,q}(Q_{1}^{+})}\), and generalizes the classical Lebesgue space \(L^{p}(Q_{1}^{+})\) when \(p=q\). The fundamental feature of these spaces, in connection with our analysis, is that they precisely capture the parabolic scaling of the equation. Indeed, considering the rescaled function \(u_{\rho}\) defined in \eqref{7.1} and the corresponding source term \(f_{\rho}\), we have
\begin{equation*}
\|f_\rho\|_{L^{p,q}(Q_1^+)}
\leq
\rho^{1-\frac{n}{p}-\frac{2}{q}}
\|f\|_{L^{p,q}(Q_{1}^+)},
\end{equation*}
for $\rho\ll 1$, whenever 
\[
\sigma(n,p,q)\defeq \frac{n}{p}+\frac{2}{q}\leq 1.
\]
Under suitable modifications of our approach, one can extend the regularity theory developed in Sections \ref{Section3}, \ref{Section4}, and \ref{Section6} to this anisotropic setting. More precisely, by replacing the assumption on the source term in the \(F\)-caloric approximation Lemma \ref{Approx} with the corresponding \(L^{p,q}\)-condition, the arguments presented in those sections remain essentially unchanged and yield the analogous regularity results in this framework, according of the values of $\sigma(n,p,q)$. We summarize these conclusions in Table \ref{tab:extesion} below. In the table, the case $\sigma(n,p,q)=0$ should be understood as corresponding to the space $\mathrm{BMO}$.
\begin{table}[h!]
\centering
\begin{tabular}{|c|c|c|}
\hline
\textrm{$\sigma(n,p,q)$} &\textrm{Boundary data} & \textrm{Regularity of $u$} \\
\hline
$\sigma(n,p,q)=1$ & $\beta,g\in C^{\alpha,\frac{\alpha}{2}}(\overline{Q_{1}^{*}})$&
$u \in C^{0,\mathrm{Log\text{-}Lip}}_{\mathrm{loc}}(\overline{Q^{+}_{1}})$ \\
\hline
$0<\sigma(n,p,q)<1$& $\beta,g\in C^{\alpha,\frac{\alpha}{2}}(\overline{Q_{1}^{*}})$&
$u \in C^{1,\min\{\alpha^{-},\,1-\sigma(n,p,q)\}}_{\mathrm{loc}}(\overline{Q^{+}_{1}})$ \\
\hline
$\sigma(n,p,q)=0$ &$\beta,g\in C^{1+\alpha,\frac{1+\alpha}{2}}(\overline{Q_{1}^{*}})$&
$u \in C^{1,\mathrm{Log\text{-}Lip}}_{\mathrm{loc}}(\overline{Q^{+}_{1}})$ \\
\hline
\end{tabular}
\medskip
\caption{Regularity results in the framework of anisotropic Lebesgue spaces}
\label{tab:extesion}
\end{table}

This demonstrates that the approach developed in this work is sufficiently flexible to accommodate other functional spaces, a direction that will be explored in future work.





\end{document}